\documentclass[a4paper,12pt]{amsart}
\usepackage{amsmath,amsfonts,amssymb,amsthm,mathtools}
\usepackage{graphics}
\usepackage{epsfig}
\usepackage{esint}
\usepackage{shuffle}
\usepackage{istgame}
\usepackage[hidelinks,hyperindex,breaklinks]{hyperref}

\usepackage{vmargin}
\setpapersize{A4}
\setmargins{2.5cm}{1.5cm}{16.5cm}{23.42cm}{10pt}{1cm}{0pt}{2cm}

\usepackage{enumerate}

\newcommand{\E}{\mathbb{E}}
\DeclareMathOperator{\tr}{\rm tr}
\DeclareMathOperator{\adj}{\rm adj}

\newtheorem{definition}{Definition}[section]
\newtheorem{proposition}[definition]{Proposition}
\newtheorem{theorem}[definition]{Theorem}
\newtheorem{remark}[definition]{Remark}
\newtheorem{lemma}[definition]{Lemma}

\begin{document}

\title{A critical drift-diffusion equation: intermittent behavior}
\author[F. Otto]{Felix Otto}
\email[F. Otto]{felix.otto@mis.mpg.de}
\author[C. Wagner]{Christian Wagner}
\email[C. Wagner]{christian.wagner@mis.mpg.de}

\maketitle

\begin{abstract}
We consider a drift-diffusion process with a time-independent and divergence-free random drift 
that is of white-noise character.
We are interested in the critical case of two space dimensions, where one has to impose a 
small-scale cut-off for well-posedness, and is interested in the marginally 
super-diffusive behavior on large scales.

\medskip

In the presence of an (artificial) large-scale cut-off at scale L, 
as a consequence of standard stochastic homogenization theory, there exist harmonic coordinates 
with a stationary gradient $F_L$; the merit of these coordinates being that
under their lens, the drift-diffusion process turns into a martingale.

\medskip

It has recently been established that the second moments diverge as 
$\mathbb{E}|F_L|^2\sim\sqrt{\ln L}$ for $L\uparrow\infty$.
We quantitatively show that in this limit, and in the regime of small P\'eclet number,
$|F_L|^2/\mathbb{E}|F_L|^2$ is not equi-integrable, 
and that $\mathbb{E}|{\rm det}F_L|/\mathbb{E}|F_L|^2 $ is small. 
Hence the Jacobian matrix of the harmonic coordinates is very peaked and non-conformal.

\medskip

We establish this asymptotic behavior by characterizing a proxy $\tilde F_L$ 
introduced in previous work as the solution of an It\^{o} SDE w.~r.~t.~the variable $\ln L$, 
and which implements the concept of a scale-by-scale homogenization 
based on a variance decomposition and
admits an efficient calculus. For this proxy, 
we establish $\mathbb{E}|\tilde F_L|^4\gg(\mathbb{E}|\tilde F_L|^2)^2$ 
and $\mathbb{E}({\rm det}\tilde F_L-1)^2\ll 1$. 
In view of the former property,
we assimilate this phenomenon to intermittency. In fact, $\tilde F_L$
behaves like a tensorial stochastic exponential, and as a field can be assimilated
to multiplicative Gaussian chaos.
\end{abstract}

\tableofcontents

\section{Introduction}

\subsection{Drift-diffusion process $X$, Gaussian ensembles of divergence-free drifts $b$, 
and critical dimension $d=2$ for decorrelated ensembles}

We are interested in a drift-diffusion process in $d$-dimensional space,
\begin{align}\label{cw01}
dX_t=b(X_t)dt+\sqrt{2}dW_t\quad\mbox{and}\quad X_{t=0}=0
\end{align}
with a time-independent and divergence-free drift $b$. The factor of $\sqrt{2}$ is chosen such that
the expectation of an observable $u$ evolves according to $\frac{d}{dt}\mathbb{E}u(X_t)$
$=\mathbb{E}{\mathcal L}u(X_t)$, where
\begin{align}\label{cw03}
{\mathcal L}u:=b\cdot\nabla u+\triangle u
\end{align}
is the generator of the Markov process $X$. 
The assumption $\nabla\cdot b=0$ ensures that the uniform distribution
$\rho=const$ is in the null space of the adjoint
${\mathcal L}^*\rho=-\nabla\cdot\rho b+\triangle\rho$.
This in particular rules out traps, i.~e.~local configurations of $b$ that make it unlikely
for $X$ to escape within moderate time.

\medskip

We are interested in vector fields $b$ 
that are randomly chosen from a stationary and centered Gaussian ensemble.
Such an ensemble is characterized by the (positive semi-definite)
Fourier transform\footnote{where we use the physicist's normalization 
$\mathcal{F}c(k)=\int{\frac{dx}{\sqrt{2\pi}^d}} 
\exp ( - i k \cdot x )c(x)$} 
${\mathcal F}C(k)$ of the covariance tensor 
$C(x-y)=\mathbb{E}b(x)\otimes b(y)$.
Loosely speaking, the most decorrelated choice of ensemble is given by
%
\begin{align}\label{cw02}
{\mathcal F}C(k)=\varepsilon^2({\rm id}-\frac{k\otimes k}{|k|^2}),
\end{align}
which can be interpreted as the Leray projection\footnote{the orthogonal projection
onto the space of a divergence-free fields w.~r.~t.~$L^2(\mathbb{R}^d)^d$,
which is the Cameron-Martin space of vectorial white noise} of vectorial white noise,
and where we give ourselves the freedom of choosing some amplitude $\varepsilon\ge 0$. 

\medskip

Does this process share the parabolic scale invariance in law
of the Wiener process $W$, namely $\lambda W_t=_{law}W_{\lambda^2 t}$?
We note that we have\footnote{In this context, ``annealed law'' refers to the expectation 
w.~r.~t.~both the environment $b$ and the thermal noise $W$.} 
$\lambda X_t=_{annealed\;law}X_{\lambda^2 t}$
provided $b$ transforms in law like a velocity under this parabolic rescaling, that is, 
$b(\lambda x)$ $=_{law}\frac{\lambda}{\lambda^2} b(x)$ $=\lambda^{-1} b(x)$.
On the other hand, it follows from (\ref{cw02}) that $b$ scales with the inverse square-root
of the volume\footnote{Treating $\varepsilon^2$ as non-dimensional (as will soon be justified),
${\mathcal F}C$ is non-dimensional, hence $C$ has units of inverse volume, and thus
$b$ has units of its square-root.}, that is, 
$b(\lambda x)$ $=_{law}\frac{1}{\sqrt{\lambda^d}} b(x)$ $=_{law}\lambda^{-d/2} b(x)$.
Hence we learn that for the ensemble (\ref{cw02}), $d=2$ is a critical dimension.

\medskip

However, for $d\ge 2$, the ensemble (\ref{cw02}) of $b$'s is too rough to give a sense to
the generator (\ref{cw03}). Indeed, the above scaling in law indicates that every
realization of $b$ is a Schwartz distribution of regularity no better than $-\frac{d}{2}$; in fact,
by the nature of the Kolmogorov criterion, its quenched regularity is slightly worse\footnote{
like Brownian motion $W$ has a H\"older regularity slightly worse than $\frac{1}{2}$},
as indicated by writing $-\frac{d}{2}-$. Hence even if the function $u$ were smooth,
$b\cdot\nabla u$ would have typical regularity of $-\frac{d}{2}-$.
As a consequence, a typical solution of ${\mathcal L}u=0$ would have at best regularity 
of $2-\frac{d}{2}-$.
Even appealing to $\nabla\cdot b=0$ to rewrite $b\cdot\nabla u$ $=\nabla\cdot ub$,
these regularities would not be sufficient to give a sense to the product $ub$, since the sum
$2-d-$ of the regularities of the factors, $2-\frac{d}{2}-$ and $-\frac{d}{2}-$, 
is negative. Also this argument reveals the criticality of $d=2$.

\subsection{UV cut-off, P\'eclet number $\varepsilon$, annealed second moments,
and super-diffusive behavior for $d=2$} 

In view of the preceding discussion, 
the ensemble (\ref{cw02}) has to be regularized for (\ref{cw01}) to make sense.
The simplest way is to introduce an ultra-violet (UV) cut-off scale. W.~l.~o.~g.~we 
may non-dimensionalize s.~t.~this scale is unity, so that (\ref{cw02})
is replaced by\footnote{In our notation, $I(A)\in\{0,1\}$ vanishes unless the statement $A$ holds.} 
\begin{align}\label{cw04}
{\mathcal F}C(k)=\varepsilon^2I(|k|\le 1)({\rm id}-\frac{k\otimes k}{|k|^2}).
\end{align}
As a consequence, every realization $b$ is smooth; more precisely,
it varies little on length scales $\ll 1$ and has a typical amplitude of order $\varepsilon$. 
Now that $b$ has acquired a characteristic length scale, we may associate a characteristic 
time scale to both diffusion and convection. 
In view of the unit diffusivity, and the convection with a drift of 
order $\varepsilon$, these characteristic time scales are of the order $1$ and $\varepsilon^{-1}$, 
respectively. Hence the parameter $\varepsilon$ plays the role of the P\'eclet number,
which is defined to be the ratio of both time scales. The P\'eclet number is the 
traditional measure of the relative strength of convection w.~r.~t.~diffusion.

\medskip

The presence of the UV cut-off breaks any scale invariance in law. 
It introduces non-universal behavior on length scales $\lesssim 1$ so that
in order to see whether $X$ behaves like a diffusion, we should consider $t\gg 1$.
The simplest question is to monitor the annealed second moments
$\mathbb{E}(\xi\cdot X_t)^2$ in some direction $\xi\in\mathbb{R}^d$ with $ | \xi | = 1 $;
the word ``annealed'' is to indicate that  the expectation is taken both with respect
to the thermal noise $W$ and the environmental noise $b$. Diffusive behavior means that
$\mathbb{E}(\xi\cdot X_t)^2$ scales as Brownian motion
$\mathbb{E}(\xi\cdot W_t)^2=t$.

\medskip

We digress by giving a heuristic argument that for the regularized ensemble (\ref{cw04}),
diffusive behavior of $X$ is (only) to be expected for $d>2$. 
Arguing by self-consistency, if $X$ behaves diffusively, we may think of it 
as a random walk on a unit lattice, unity being the characteristic length of $b$ 
discussed above (and the only characteristic length scale present). 
Such a random walk is transient\footnote{meaning that it does not return to the same 
lattice point after some time} iff $d>2$.
Since in view of (\ref{cw04}), $b$ decorrelates over distances $\gg 1$,
the contribution $b(X_t)dt$ in (\ref{cw01}) thus decorrelates over time. 
Hence the drift acts as a (statistically independent) diffusion, 
increasing the overall (effective) diffusivity.
This indeed can be made rigorous by the approach of Subsection \ref{ss:hom},
as done in \cite{KozmaToth17,Toth18}. 

\medskip 
 
Let us return to the critical dimension $d=2$ and
the asymptotic behavior of annealed second moments $\mathbb{E}(\xi\cdot X_t)^2$,
this time for the ensemble (\ref{cw04}).
In the physics literature,~cf.~\cite{FisherEtAl}, the (borderline) super-diffusive behavior
\begin{equation}\label{in12}
\mathbb{E}(\xi\cdot X_t)^2\sim t\sqrt{ \ln t } \quad { \rm as } ~ t \rightarrow \infty
\end{equation}
was informally derived by a renormalization group argument, 
which is reproduced in \cite[Appendix]{CMOW}. A first rigorous proof of some super-diffusive
behavior was given in \cite{TothValko12}. Up to double logarithms and modulo the Laplace
transform, the scaling (\ref{in12}) was
rigorously established in \cite{CannizaroHaunschmidtSibitzToninelli}.
The double logarithms and the Laplace transform were eliminated in \cite{CMOW}.
Moreover, precise asymptotics were identified in the regime of small P\'eclet number:
\begin{equation}\label{cw11}
\frac{\mathbb{E}(\xi\cdot X_t)^2}{2 t\sqrt{1+\frac{1}{2}\varepsilon^2\ln t}} 
\approx 1\quad\mbox{for all}\;t \ge 1
\quad\mbox{provided\footnotemark}\;\varepsilon\ll 1.
\end{equation}
\footnotetext{By the statement $A\approx B$
for $\varepsilon\ll 1$ we mean that that for every $\delta>0$ there exists an $\varepsilon_0>0$
such that $|\frac{A}{B}-1|\le\delta$ provided $\varepsilon \le \varepsilon_0$.}

\medskip

The method of \cite{TothValko12,CannizaroHaunschmidtSibitzToninelli} 
lifts the analysis to the infinite-dimensional 
Markov process of the environment as seen from the particle $X$ (as the jargon goes), 
which is induced by the generator (\ref{cw03}) when replacing spatial derivatives by what 
occasionally is called horizontal derivatives.
This has the advantage of naturally extending to a time-dependent Markovian environment 
\cite{DeLimaFeltesWeber24}. Based on the Gaussian nature of the underlying (white) noise,
the analysis of the generator relies on
the Wiener chaos decomposition of observables. This functional-analytic approach has
also been employed for stochastic PDE that define a (field-valued) stationary Markov process. 
Next to the Gaussianity of the underlying stationary measure on configuration space,
it relies on the control of the skew-symmetric part of the generator
by its symmetric one. This has been carried out for the stochastic Burgers equation\footnote{the 
KPZ equation formulated in terms of the slope} in \cite{MR4168394}.   
It has also been implemented in scaling-critical cases,
like the anisotropic KPZ equation in dimension $d=2$, see \cite{MR4642815}. 
Very recently, \cite{ABK24} reports a quenched version of (\ref{cw11})
valid for all $\varepsilon$. The method is closer to ours, see Subsection \ref{ss:scale} for a (short)
discussion.


\subsection{Intermittency of Lagrangian coordinates on large time scales}

Let us consider 
\begin{align}\label{so06}
\bar X_t(x):=E[ X_t ] \quad\mbox{where $X$ solves}\quad X_t= x + \int_0^t ds \, b(X_s)+\sqrt{2}W_t
\end{align}
and where $E[ \cdot ] $ denotes the expectation w.~r.~t.~Brownian motion $W$. 
Hence $\bar X_t(x)$ is the position of a particle at time $t$ that started at $x$,
averaged w.~r.~t.~to the thermal noise $W$.
We like to think of $\bar X$ as Lagrangian coordinates; in fact, $\bar X$ arises from
considering the random dynamical system where all particles are subject to the same
thermal noise (and of course the same environment $b$), 
and then taking the expectation w.~r.~t.~the thermal noise. 

\medskip

We are interested in the Jacobian matrix $\nabla\bar X_t$ of these coordinates, 
which is a stationary tensor field, and monitor its squared Frobenius norm $|\nabla\bar X_t|^2$.
A version of (\ref{cw11}) reads\footnote{In this text, it can be inferred from
(\ref{so13}) and (\ref{mr05}) via definition (\ref{tl01}).} 
\begin{align}\label{so11}
\mathbb{E}\int_T^{2T}dt|\nabla\bar X_t|^2\approx 2
T \sqrt{1+{\textstyle\frac{1}{2}}\varepsilon^2\ln T}
\quad\mbox{for}\quad T\ge 1\;\mbox{and}\;\varepsilon\ll 1.
\end{align}
The main insight of this work is that $|\nabla\bar X_t|^2$ and even its time average
$\int_T^{2T}dt|\nabla\bar X_t|^2$ develop high peaks,
in the sense that when normalized, it is not equi-integrable w.~r.~t.~the ensemble:

\begin{theorem}\label{th:forFabio}We have for $ \varepsilon \ll 1 $ and $ \varepsilon^2 \ln T \gg 1 $
\begin{align}\label{ff03}
\E \int_{T}^{2T} dt | \det \nabla \bar X_t | \ll \E \int_T^{2T} dt | \nabla \bar X_t |^2
\end{align}
and
\begin{align}\label{so10}
\frac{\int_T^{2T}dt|\nabla\bar X_t|^2}{\mathbb{E}\int_T^{2T}dt|\nabla\bar X_t|^2}\quad
\mbox{is not equi-integrable}.
\end{align}
\end{theorem}

In fact, we obtain a more quantitative statement, see (\ref{ao13}) for a related quantity. 
We stress that this phenomenon of deterioration of the Lagrangian coordinates
can only be observed on sufficiently large time scales: 
More precisely, $\ln t$ not only has to be large w.~r.~t.~$\varepsilon^{-2}$ but it has to vary
substantially over the period of observation\footnote{An example of a sufficiently large
time interval is $t\in(\exp(\varepsilon^{-\alpha}),\exp(\varepsilon^{-\beta}))$ for $\varepsilon\ll 1$ and $ \beta > \alpha > 2 $.}. 
In fact, we also have\footnote{as a side effect of (\ref{so13}) and (\ref{mr05})}
\begin{align}\label{so12}
\mathbb{E}\int_T^{2T}dt |\nabla(\bar X_t-\fint_T^{2T}ds \bar X_s)|^2\lesssim
\varepsilon^2\mathbb{E}\int_T^{2T}dt|\nabla\bar X_t|^2,
\end{align}
which shows that to leading order in $\varepsilon\ll 1$ and in an $L^2_t$-averaged sense,
$\nabla\bar X_t$ does not depend on $t$ on any time range $(T,MT]$ with large but fixed $M$.
The discussion around (\ref{so14}) indicates that this phenomenon takes place
on spatial scales $\ell$ that are exponentially large, too: $\ln\ell\gg\varepsilon^{-2}$ . 
Incidentally, the Levy area does not seem to be a quantity where this phenomenon is
more easily detectable.

\medskip

The proof of Theorem \ref{th:forFabio} is based on the standard representation
$\xi\cdot\bar X_T(x)=u(t=0,y=x)$, where $u(t,y)$ solves the 
backwards Kolmogorov equation $-\partial_tu-{\mathcal L}u=0$
with terminal data $u(t=T,y)=\xi\cdot y$. We rewrite this as
\begin{align}\label{so09}
\xi\cdot(\bar X_t(x)-x)=v(t,x)\quad\mbox{so that}\quad\nabla(\xi\cdot\bar X_t)
=\xi+\nabla v(t,\cdot)
\end{align}
where $v(t,x) = u(T-t,x) - \xi \cdot x $ is the solution of
\begin{align}\label{so07}
\partial_tv-{\mathcal L}v=\xi\cdot b\quad\mbox{and}\quad v(t=0)=0,
\end{align}
and thus has the advantage of being stationary. The second identity in (\ref{so09}) 
translates (\ref{so11}), (\ref{so10}),
and (\ref{so12}) into statements on the solution $v$ of the PDE (\ref{so07}).

\subsection{Stream function $\psi$, elliptic coefficient field $a$, 
harmonic coordinates, corrector $\phi$}\label{ss:hom}

The method for (\ref{cw11}) from \cite{CMOW}, on which this paper is based, is much more specific to
the drift-diffusion process (\ref{cw01}). It appeals to stochastic homogenization theory for
elliptic equations in divergence form, which we review now. To the best of our knowledge, 
\cite{CMOW} was the first work to implement scale-by-scale homogenization reminiscent
of the informal renormalization group argument in \cite{FisherEtAl} for the critical problem
at hand.
The condition $\nabla\cdot b=0$ means that
the vector field $b$, seen as a $(d-1)$-form, is closed. Hence there exists a $(d-2)$-form
$\psi$ such that $b$ can be written as the exterior derivative of $\psi$. 
As a consequence of (exterior) calculus, we may represent the generator
${\mathcal L}$ as a purely second-order operator in divergence form:
\begin{align}\label{cw09}
{\mathcal L}u=\nabla\cdot a\nabla u\quad\mbox{where}\quad a:={\rm id}+\psi;
\end{align}
note that (the non-symmetric) $a$ is uniformly positive definite, 
so that ${\mathcal L}$ is elliptic.
In the case of $d=2$, $\psi$ can be identified with a scalar known as stream function, 
and (\ref{cw09}) turns into
\begin{align}\label{cw06}
b=J\nabla\psi\quad\mbox{and}\quad a:={\rm id}+\psi J\quad
\mbox{where}\quad J:=\left(\begin{array}{rr}0&-1\\1&0\end{array}\right).
\end{align}
%
%
%

\medskip

It can be easily seen that iff $d>2$, the stationary and centered Gaussian ensemble (\ref{cw04})
of $b$'s is such that there exists 
a stationary and centered Gaussian ensemble of skew symmetric matrices $\psi$.
Then it follows from qualitative homogenization theory for
stationary and ergodic ensembles of uniformly elliptic coefficient fields $a$, 
as initiated in \cite{Kozlov79,PapanicolaouVaradhan81}, 
and adapted to the present case in \cite{Oelschlaeger88},
that $X$ displays diffusive behavior. More precisely, there exists
a $\lambda\in(0,\infty)$ such that for all $\xi$
\begin{align}\label{cw08}
\lambda|\xi|^2=\lim_{t\uparrow\infty}\frac{\mathbb{E}(\xi\cdot X_t)^2}{2t}.
\end{align}
The statement (\ref{cw08}) remains valid almost surely in $b$ 
when only the expectation w.~r.~t.~$W$ is taken,
as a consequence of what is called quenched invariance principle.

\medskip

In order to establish (\ref{cw08}), one deforms
the affine coordinate $x\mapsto\xi\cdot x$ into what is called a harmonic coordinate $u$
solving ${\mathcal L}u=0$. The merit of $u$ is that $u(X) $ becomes a martingale:
\begin{align}\label{cw16}
d u(X) = \nabla u(X) \cdot dW,
\end{align}
which follows from an application of It\^{o} calculus to (\ref{cw01}).
Making the Ansatz $u(x)=\xi\cdot x+\phi(x)$, this leads to the 
PDE $ \nabla\cdot a(\xi+\nabla\phi)=0$ for what is called the corrector $\phi$.
By an adaptation of the Lax-Milgram argument, see \cite{Fehrman23} for a modern presentation,
the above-mentioned theory establishes
that there exists a unique stationary gradient field $\nabla\phi$ of finite second moments with
\begin{align}\label{cw05}
\nabla\cdot a(\xi+\nabla\phi)=0\quad\mbox{and}\quad\mathbb{E}\nabla\phi=0;
\end{align}
by stationarity of $\nabla\phi$ we may suppress the argument $x$ in
$\mathbb{E}\nabla\phi=0$.
This theory also gives a representation for $\lambda$ in (\ref{cw08}) as the expectation
of the (stationary) flux $a(\xi+\nabla\phi)$
\begin{align}\label{cw10}
\lambda\xi=\mathbb{E}a(\xi+\nabla \phi).
\end{align}
With help of (\ref{cw05}), (\ref{cw10}) can be reformulated as
$\lambda|\xi|^2$ $=\mathbb{E}(\xi+\nabla\phi)\cdot a(\xi+\nabla\phi)$. 
In view of the form (\ref{cw09}), and the second item in (\ref{cw05}), this 
yields 
\begin{align}\label{cw17}
\lambda=\mathbb{E}|\xi+\nabla\phi|^2=1+\mathbb{E}|\nabla\phi|^2\quad\mbox{provided}\;|\xi|=1.
\end{align}
In particular, a divergence-free
drift always increases the diffusivity. The stochastic homogenization theory for a 
non-divergence-free drift is much more subtle, and only established for $d>2$ and
$\varepsilon\ll 1$, see \cite{BricmontKupiainen91} for a somewhat informal and
\cite{SznitmanZeitouni06} for the rigorous proof.

\subsection{IR cut-off $L$, Gaussian free field, asymptotics of $\lambda_L$} 
As explained, in $d=2$, the ensemble (\ref{cw04}) does not allow for a
straightforward application of stochastic homogenization. In \cite{Fannjiang98},
it was suggested to quantify the super-diffusivity by monitoring the scaling
of the effective diffusivity $\lambda_L$ in an artificial large-scale (infra-red) cut-off $L$.
This approach was carried out in \cite{CMOW} in the following way:
The infra-red (IR) cut-off is implemented on the level of the ensemble by replacing 
(\ref{cw04}) by
\begin{align}\label{cw15}
{\mathcal F}C_L(k)=\varepsilon^2I(L^{-1}<|k|\le 1)({\rm id}-\frac{k\otimes k}{|k|^2}).
\end{align}
Since $|k|^2 { \rm id } -k\otimes k$ $=Jk\otimes Jk$,
this ensemble of drifts $b_L$ arises from a (stationary and centered) 
ensemble of stream functions $\psi_L$ characterized by its (scalar) covariance function 
$c_L(x-y)=\mathbb{E}\psi_L(x)\psi_L(y)$ given by its Fourier transform
\begin{align}\label{cw19}
{\mathcal F}c_L(k)=\varepsilon^2I(L^{-1}<|k|\le 1)\frac{1}{|k|^2}.
\end{align}
This ensemble is known as the Gaussian free field, with UV and IR cut-off. The necessity
of both cut-offs for stationarity becomes obvious from the logarithmic divergence in
\begin{align}\label{cw20}
\mathbb{E}\psi^2_L=c_L(0)=\varepsilon^2\int\frac{dk}{2\pi}{\mathcal F}c_L(k)
=\varepsilon^2\ln L.
\end{align}

\medskip

In the presence of a stationary ensemble of stream functions,
stochastic homogenization applies, see Subsection \ref{ss:hom}.
Setting
\begin{align}\label{ho15}
a_L:={\rm id}+\psi_L J,
\end{align}
cf.~(\ref{cw06}), there exists a unique stationary gradient field\footnote{There is
no reason the believe that $\phi_L$ itself is stationary.} $\nabla\phi_L$  
of finite second moments with
\begin{align}\label{cw14}
\nabla\cdot a_L(\xi+\nabla\phi_L)=0\quad\mbox{and}\quad\mathbb{E}\nabla\phi_L=0,
\end{align}
cf.~(\ref{cw05}). Monitoring
\begin{align*}
\lambda_L\xi=\mathbb{E}a_L(\xi+\nabla \phi_L),
\end{align*}
cf.~(\ref{cw10}), it is established in \cite{CMOW} that
\begin{align}\label{cw13}
\lambda_L\approx \sqrt{1+\varepsilon^2\ln L}
\quad\mbox{provided}\;\varepsilon\ll 1. 
\end{align}
Loosely speaking, in \cite{CMOW}, (\ref{cw11}) is derived from (\ref{cw13}) via
approximating $\mathbb{E}(\xi\cdot X_t)^2$ by $\lambda_L t$ for $L=\sqrt{t}$.


\section{Statement of result on Jacobian matrix $F={\rm id}+\nabla\phi$
of harmonic coordinates}

Let us denote by $\nabla\phi_L^i$ the solution of (\ref{cw14}) with $\xi=e^i$,
where $e^i$ denotes the Cartesian vector in direction $i=1,2$.
In this paper, we analyze the harmonic coordinates $x\mapsto u_L(x):=x+(\phi_L^1(x),\phi_L^2(x))$
of the cut-off ensemble (\ref{cw15}). More precisely, we study its stationary
Jacobian matrix field
\begin{align}\label{cw17b}
F_L:=\nabla u_L={\rm id}+(\nabla\phi_L^1,\nabla\phi_L^2)
\end{align}
%
in the limit $L\uparrow\infty$ and regime $\varepsilon\ll 1$. We already know from
(\ref{cw17}) and (\ref{cw13}) that
\begin{align}\label{cw18}
\mathbb{E}|F_L|^2=2\lambda_L\approx 2\sqrt{1+\varepsilon^2\ln L}
\quad\mbox{for}\;\varepsilon\ll 1,
\end{align}
where $|F_L|^2={\rm tr}F_L^*F_L$ denotes the squared Frobenius norm of $F_L$.

\medskip

Analogous to (\ref{cw16}), 
the harmonic coordinates $x\mapsto u_L(x)$ 
convert the process $X_L$, defined via (\ref{cw01}) with $b$ replaced by $b_L$, 
into a two-component martingale.
In view of (\ref{cw16}), the quadratic expression $ F_L^*F_L $ of its Jacobian matrix 
$F_L=\nabla u_L$ governs the quadratic variation of this martingale. 
We thus consider the study of the asymptotics of $F_L$ as a first step towards 
understanding the finer structure of the process $X$.

\medskip


\begin{theorem}[Intermittency]\label{thm:intermittency}
For any $\alpha<\frac{1}{2}$, provided\footnote{
The regime of $\varepsilon$ and $L$ in which (\ref{ao17}) holds shrinks as $\alpha \uparrow \frac{1}{2}$.} $\varepsilon\ll1$ and $\lambda_L \gg 1$ 
we have\footnote{Here and in the sequel, $A\lesssim B$ means $A\le CB$ for some
universal constant $C$, i.~e.~independent of the parameters $\varepsilon$ and $L$.} 
\begin{align}
\E | \det F_L | \lesssim\varepsilon \E | F_L |^2+1 &\ll \E | F_L |^2,\label{mr02}\\
\E I( | F_L |^2\le\lambda_L^\alpha
\E | F_L |^2) | F_L |^2 &\ll \E | F_L |^2.\label{ao17}
\end{align}
\end{theorem}

Note that in view of (\ref{cw18}), $\lambda_L \gg 1$ amounts to 
$L$ being exponentially large in $\frac{1}{ \varepsilon^2 }$.

\medskip

Let us comment upon Theorem \ref{thm:intermittency}.
Its first part (\ref{mr02}) amounts to the observation that typically,
the Jacobian matrix $F_L(x)$ has two singular values of very different size.
This indicates an extreme non-conformality of the harmonic coordinates.
We note that since $\det F_L=\partial_1(u_L^1\partial_2 u_L^2)
-\partial_2(u_L^1 \partial_1 u_L^2)$ is a null-Lagrangian, (\ref{mr02}) is
only interesting because of the absolute values.
The second part (\ref{ao17})
implies that for typical point $x$, $|F_L(x)|^2$ (which is the sum of the squares
of the singular values) is much below its average (which in view of (\ref{cw18}) is $2\lambda_L$).
Hence the field $|F_L|^2$ features high peaks, suggesting an intermittent behavior of the process.
In other words, (\ref{ao17}) means that the random variable $|F_L|^2/\mathbb{E}|F_L|^2$
has bounded expectation but is not equi-integrable as $L\uparrow\infty$.

\medskip

Let us now interpret Theorem \ref{thm:intermittency}.
It is consistent with the following picture of the process $X$:
Since a realization $\psi$ of the Gaussian free field with UV cut-off
has closed stream lines, the process (\ref{cw01}) would be trapped in a finite region 
if it were not for diffusion. Incidentally, this
collection of loops is an object of active research 
in view of its emergent conformal invariance.
Diffusion allows the process $X$ to move from one level set to the next.
As $X$ explores larger and larger neighborhoods of its starting point,
it occasionally hits loops of large diameter, which introduces a ballistic element.
In the absence of an IR cut-off, there are loops of larger and larger diameter,
which is at the origin of super-diffusivity. This picture suggests an intermittent behavior
of $X$, with bursts of ballistic motion interrupting longer portions of more
diffusive behavior. The burst-like character is mirrored by non-equi-integrability (\ref{ao17}),
the ballistic nature is reflected by the dominance of one singular value (\ref{mr02}).
Thus Theorem \ref{thm:intermittency} 
is in contrast (but not in contradiction) with the quenched invariance
principle for $X$ reported in \cite{ABK24}. Theorem \ref{thm:intermittency}
seems to capture more of
the fine structure of the level sets of the Gaussian free field.


\section{Strategy of proof}\label{sec:proof-strategy}

\subsection{Variance decomposition of $\psi$ into a martingale
in $\tilde\lambda^2=1+\varepsilon^2\ln L$}\label{ss:scale}
Rather than estimating $F={\rm id}+\nabla\phi_L$ directly, we use a stationary proxy 
$\tilde\phi_L$ that is built on the idea of a scale-by-scale homogenization, 
see the next Subsection \ref{ss:defproxy}.
In the mathematical community,
the idea of a scale-by-scale homogenization for drift-diffusion processes with random drift
seems to originate in \cite{BricmontKupiainen91}. 
In an setting of anomalous diffusion for such processes due to failing homogenization, 
this strategy was first implemented in \cite{CMOW} for the problem at hand.
Shortly afterwards, scale-by-scale homogenization has also been employed for the construction
of (time-dependent) drifts that produce anomalous dissipation 
\cite{ArmstrongVicol23,BurczakSzekelyhidiWu23}, and very recently has been refined
in \cite{ABK24} to a quenched invariance principle for the problem at hand.

\medskip

In this paper, we follow the continuum version of the scale-by-scale
homogenization \cite{CMOW} very recently introduced in \cite{SuppCMOW}.
This construction relies on the fact that the family of stationary centered Gaussian fields 
$\psi_L$ characterized by (\ref{cw19}) can be coupled to form a process $L\mapsto\psi_L$
with independent increments,~i.~e.~
\begin{equation}\label{in07}
\psi_{L_+} - \psi_{L} ~ \text{is~independent~of} ~ \psi_{L} 
\end{equation}
for any $ L_+ \ge L $. Note that \eqref{in07} implies
$$
\E ( \psi_{L_+} - \psi_L ) (x) ( \psi_{L_+} - \psi_L ) (y) 
= \E  \psi_{L_+} (x) \psi_{L_+} (y) - \E \psi_L (x) \psi_L (y)
$$
so that by \eqref{cw19} the increments have covariance characterized by
\begin{equation}\label{in08}
\mathcal{F} c_{ \psi_{L+} - \psi_L } (k) = \varepsilon^2 I ( L_+^{-1} <
|k|\le L^{-1} ) \frac{1}{|k|^2}.
\end{equation}
This amounts to a variance decomposition of $\psi$.

\medskip

In view of (\ref{cw20}), it is natural to replace the independent variable $L\in[1,\infty)$ by
$\tilde\lambda^2-1\in[0,\infty)$, where
\begin{equation}\label{tl01}
\tilde\lambda^2 \coloneq 1 + \varepsilon^2 \ln L
\end{equation}
so that $\E \psi^2$ $=\tilde\lambda^2-1$; the notation (\ref{tl01}) is motivated by (\ref{cw13}).
Hence the martingale $\tilde\lambda^2-1\mapsto\psi_L$ shares features with Brownian motion;
we denote its distributional derivative\footnote{which can also be interpreted as 
It\^{o} differential}
by $ d \psi $. 
Due to the independence of increments and \eqref{in08}, 
the differential\footnote{\label{footnote:quadvar}That is, we have
$$
\int_{\ell_1}^{\ell_2} [ d \psi \, d \psi ]
\coloneq \lim_{ { \rm fineness }( \{ L_k \} ) \rightarrow 0 } \sum_{ k } ( \psi_{ L_{k} }(x) - \psi_{ L_{k-1} }(x) )^2,
$$
where the limit is taken over partitions of $ [\ell_1, \ell_2] $ with mesh size tending to $ 0 $. 
It is well know, that by Gaussianity, the limit is deterministic, and hence by stationarity 
also independent of $ x $. For a more detailed description, we refer the reader to the 
Sections 4.1 and 4.3 in \cite{LeGall}, with the warning that a different notation for the 
quadratic variation is used within that reference.} of the quadratic variation of $ \psi $ 
is given by
\begin{equation}\label{in09}
[ d \psi \, d\psi] = d \tilde\lambda^2.
\end{equation}
Relation (\ref{in09}) does not contain the explicit information
on the full quadratic variation, since both copies of 
$\psi$ are (implicitly) evaluated at the same point $x$.
Information on the covariation at different points is for instance contained in (\ref{qvarpsi})
below. In fact, Lemma \ref{lem:quad-var-nabla-phi} and Lemma \ref{lem:quad-var-nabla-nabla-phi}
can be seen as characterizing $[d\psi \, Ad\psi]$ for specific pseudo-differential operators $A$.

\subsection{Definition of the proxy corrector $\tilde\phi$ via an It\^{o} SDE in $\tilde\lambda^2$}\label{ss:defproxy}

Since our main result deals with the behavior of nonlinear coordinates, 
we follow \cite{SuppCMOW} in adopting a rather strict differential geometric language:
It is natural to interpret the coordinate directions $\xi$ as a linear form,
that is, a cotangent vector. Likewise, it is natural to interpret the differential $\nabla$ as
converting a scalar field into a cotangent vector field. 
It follows from (\ref{cw14}) that the cotangent vector field $\nabla\phi$ is linear
in $\xi$, so that we interpret $\nabla\phi$ as field with values in (linear) 
endomorphisms of cotangent space. Hence (\ref{cw14}) may be rewritten as
\begin{equation}\label{in10}
\nabla \cdot a_L ( { \rm id } + \nabla \phi_L ) = 0,
\end{equation}
where the product in the flux $a_L ( { \rm id } + \nabla \phi_L )$ is the composition
of endomorphisms.
In line with this, we will interpret $\phi$ as a tangent vector field with 
$\phi=\phi^i e_i $, where $\{e_i\}_{i=1,2}$ denotes the Cartesian basis of the tangent 
space\footnote{Upper indices enumerate the components of a tangent vector, 
whereas lower indices enumerate those of a cotangent vector, 
and we employ Einstein's convention of summation
over repeated indices provided one is up and the other down.} 
of the plane. 
On the other hand, the matrix representation of $\nabla\phi$ w.~r.~t.~the
(dual) Cartesian basis $\{e^j\}_{j=1,2}$ of cotangent space,
defined through $(F\xi)_j=F_j^i\xi_i$ where $\xi=\xi_je^j$, 
is given by  $(\nabla\phi)_j^i=\partial_j\phi^i$.

\medskip

We now recall the definitions from \cite[(2), (3) \& (4)]{SuppCMOW}. 
We shall often suppress the explicit dependence on $L$ 
and rather consider the objects as functions or Schwartz distributions of the
variable $\tilde\lambda^2$. The random Schwartz distributions 
(in $\tilde\lambda^2$ and in space)
$d\phi$ and $d\sigma$ with values in tangent vectors are defined\footnote{\label{footnote:dphi}$d\phi$
is not to be confused with the distributional derivative of $\phi$ w.~r.~t.~$\tilde\lambda^2$;
it is an approximation thereof, as we explain below. In the continuum approach from 
\cite{SuppCMOW} that we use here, it corresponds to the incremental quantity $ \phi' $ 
in \cite{CMOW}.} via the spatial 
Helmholtz decomposition
\begin{align}\label{wr01}
\tilde\lambda \nabla d\phi + d\psi J = J\nabla d\sigma
\quad\mbox{and}\quad
\mathbb{E} d\phi = \mathbb{E} d\sigma = 0.
\end{align}
Since by linearity, the random $d\phi$ and $d\sigma$ inherit stationarity 
in space from $d\psi$, the second part of (\ref{wr01}) ensures uniqueness\footnote{Alternatively,
the definition could be formulated in Fourier space, as done in Appendix \ref{ss:var-nabla-phi}.}. 
In Appendix \ref{ss:qv1} we derive the two local relations
\begin{align}\label{ho36}
\tr \nabla d\phi=0\quad\mbox{and}\quad
\tilde\lambda \tr J\nabla d\phi=-d\psi.
\end{align}
Since $\nabla d\phi$ and $\nabla d\sigma$ are Schwartz distributions with 
values in the endomorphisms of cotangent space, $\tilde\lambda$ is a smooth scalar,
and $J$ a constant endomorphisms, (\ref{wr01}) makes pathwise sense. 
Since $\tilde\lambda$ is deterministic, $d\phi$ and $d\sigma$ inherit the property of
having independent increments (in $\tilde\lambda^2$) from $d\psi$, which allows us to
consider them as drivers (in the jargon of SDEs).

\medskip

Equipped with these drivers, we define the proxy corrector $\tilde\phi$ and the
proxy flux corrector\footnote{in this paper, $\tilde\sigma$ does not
play a role} $\tilde\sigma$ as a process in 
$\tilde\lambda^2$ with values in stationary (smooth) tangent vector fields in $x$, 
via the Itô\footnote{As opposed to (\ref{wr01}), It\^{o} calculus is required
in (\ref{wr04}) in order to give sense to the product $\tilde\phi^i\partial_id\phi$ because
of the lack of regularity in $\tilde\lambda^2$.} differential equation
\begin{align}
d\tilde\phi&=(1+\tilde\phi^i\partial_i ) d\phi, 
\quad\tilde\phi|_{\tilde\lambda^2=1}=0,\label{wr04}\\
d\tilde\sigma&= d\sigma+\tilde\sigma^i\partial_id\phi+d\psi\,\tilde\phi
+\frac{1}{2\tilde\lambda} J\tilde\phi  \,d\tilde\lambda^2,
\quad\tilde\sigma|_{\tilde\lambda^2=1}=0.\label{wr02}
\end{align}
Note that (\ref{wr01}) implies $\tilde\lambda\triangle d\phi$ $=\nabla\cdot d\psi J$, 
an equation that amounts to the homogenization, and the linearization in $\varepsilon$, 
of the equation for the (distributional) derivative of $\phi$ w.~r.~t.~$\tilde\lambda^2$. 
Note also that the operation $1+\tilde\phi^i\partial_i$ 
amounts to taking the two-scale expansion. Hence (\ref{wr04}) implements the concept of a
scale-by-scale homogenization, on the level of the corrector itself.

\medskip

Recall from homogenization theory that the pair $(\phi,\sigma)$ 
of corrector and flux corrector is characterized 
by the Helmholtz decomposition $a({\rm id}+\nabla\phi)=\lambda{\rm id}+J\nabla\sigma$.
It was shown -- on a few pages and in a self-contained way --
in \cite[(5)]{SuppCMOW} that the residuum $f$ arising from
the approximation $(\tilde\phi,\tilde\sigma)$, defined through
\begin{align}\label{wr07}
a({\rm id}+\nabla\tilde\phi)=\tilde\lambda\,{\rm id}+J\nabla\tilde\sigma + f,
\end{align}
which is a process in $\tilde\lambda^2$ with values in stationary 
endomorphism fields in $x$,
satisfies
\begin{align}\label{wr03} 
\E f = 0
\quad { \rm and } \quad
\E | f |^2 \lesssim \varepsilon^2 \tilde\lambda,
\end{align}
see \cite[Proposition 1]{SuppCMOW}. 
We remark that (\ref{wr03})
states that the residuum $f$ is $O(\varepsilon\tilde\lambda^{-\frac{1}{2}})$ 
times smaller than the effective diffusivity $\lambda\approx\tilde\lambda$, 
to which it should be compared in view of (\ref{wr07}). 
Since in view of Lemma \ref{lem:evolution-proxy-phi}, 
the proxy corrector $\tilde\phi$ is $O(\varepsilon\tilde\lambda^{-1})$ times smaller
than linear growth $L$, the approach of \cite{CMOW}
amounts to a quantitative stochastic homogenization result
of order $\varepsilon\tilde\lambda^{-\frac{1}{2}}$ 
$\sim\varepsilon^\frac{1}{2}\ln^{-\frac{1}{4}}L$, cf.~(\ref{tl01}). 
Based on (\ref{wr07}) \& (\ref{wr03}), 
it was shown in \cite[(100)]{CMOW} that on the entire
time interval $(0,L^2)$, $\nabla v$ defined in (\ref{so07}) is approximated
by $\nabla\tilde\phi$: 
\begin{align}\label{so08}
\mathbb{E}\int_0^{L^2}dt|\nabla v-\nabla\tilde\phi|^2\lesssim\varepsilon^2\tilde\lambda L^2.
\end{align}
Apart from (\ref{wr03}) and (\ref{so08}), 
the present paper is self-contained (and reasonably short).

\medskip

A continuum variance decomposition, i.~e.~a continuum decomposition of 
an underlying Gaussian field,
often the Gaussian free field $\psi$ like here,
according to spatial scales $L$ is also used in quantum field theory
(where the underlying noise is thermal instead of environmental).
For instance, in case of the $\Phi^4_d$ model for $d=2,3,
$\cite{BG} use it to control (exponential) moments of the
field $\phi$ with help of suitable martingales in $\ln L$, analogous to $\tilde\phi$.
The variance decomposition also induces a
Hamilton-Jacobi-type evolution equation for the effective Hamiltonian in $\ln L$,
known as Polchinski equation. In \cite[Section 2.3]{GM}, to cite another recent
work, proxys for the solution of the Polchinski equation are constructed by truncation
of an expansion (analogous to an expansion in $\varepsilon$ here),
and the evolution of the residuum is monitored (like $f$ in \cite[(7)]{SuppCMOW}).


\subsection{Intermittency of the proxy Jacobian matrix $\tilde F={\rm id}+\nabla\tilde\phi$}

In line with (\ref{cw17b}), we set
\begin{align}\label{cw17bis}
\tilde F:={\rm id}+\nabla\tilde\phi;
\end{align}
following \cite[(52) \& (79)]{CMOW}, 
we shall argue in the proof of Theorem \ref{thm:intermittency} that
\begin{align}\label{wr03bis}
\E|F-\tilde F|^2\le\mathbb{E}|f|^2\stackrel{(\ref{wr03})}{\lesssim}\varepsilon^2\tilde\lambda.
\end{align}
Note also that in view of the relation (\ref{so09})
and definition (\ref{cw17bis}), the result (\ref{so08}) translates into
\begin{align}\label{so13}
\mathbb{E}\int_0^{L^2}dt|\nabla\bar X_t-\tilde F|^2\lesssim\varepsilon^2\tilde\lambda L^2.
\end{align}
Intermittency is more drastically seen on the level of the proxies:
\begin{theorem}[Intermittency for the proxies]\label{thm:proxy-intermittency}
We have
\begin{align}
\E ( \det \tilde F - 1 )^2 &\lesssim \varepsilon^2, \label{mr04} \phantom{ \frac{1}{2} } \\
\E | \tilde F |^2 &\approx 2 \tilde\lambda, \label{mr05} \phantom{ \frac{1}{2} } \\
\E | \tilde F |^4 &\approx \frac{8}{3} \tilde\lambda^3 + \frac{4}{3} \label{mr06}
\end{align}
as $\varepsilon\ll 1$ uniformly in $L\ge 1$, and thus in particular in our regime
$\tilde\lambda \gg 1 $ and $\varepsilon\ll 1$.
\end{theorem}

Clearly, the fourth-moment asymptotics (\ref{mr06}) is an
intermittency statement for the proxy $\tilde F$ of much stronger form than the statement 
(\ref{ao17}) on second-moment tails of the harmonic coordinates $ F $. 
In view of the error estimate (\ref{wr03bis}), 
which is on the level of second moment, only the estimate in the
following proposition transfers from $\tilde F$ to $F$.

\begin{proposition}\label{prop:proxy-intermittency-2}
We have
\begin{align}\label{ao13}
\E I( | \tilde F |^2 \leq \hat r \, \E | \tilde F |^2) | \tilde F |^2 \ll \E | \tilde F |^2
\end{align}
provided $ \varepsilon^2 \ll 1$, $ \tilde\lambda \gg 1$ and
\begin{align}\label{ao12}
\hat r \ll  e^{ - \sqrt{2\ln \tilde\lambda}} \sqrt{\tilde\lambda} .
\end{align}
\end{proposition}

The statement \eqref{ao13} is not an immediate consequence of
Theorem \ref{thm:proxy-intermittency} but follows from an analysis
of the backward Kolmogorov equation for $|\tilde F|^2$. We will derive the SDE for
$|\tilde F|^2$ in the proof of (\ref{wr42}).
\medskip

The strategy that leads to Theorem \ref{thm:proxy-intermittency} is in the spirit of 
\cite{CMOW} (in its continuum version \cite{SuppCMOW}): 
It is based on the It\^{o} SDE
\begin{align}\label{wr36}
d \tilde F = \tilde F \nabla d\phi + \tilde\phi^i\nabla\partial_id\phi,
\quad\tilde F|_{\tilde\lambda^2=1}={\rm id},
\end{align}
which follows by applying $\nabla$ to (\ref{wr04}),
and noting that
$\partial_i d\phi\otimes\nabla\tilde\phi^i$ 
$= ( \nabla\tilde\phi ) \nabla d\phi$,
where the product on the r.~h.~s.~denotes the composition of endomorphisms
(of cotangent space).
In particular, $\tilde F$ is a martingale.
First, we use It\^{o} calculus
on (\ref{wr36}) to derive SDEs on $\det\tilde F$ and $|\tilde F|^2$ 
as functions of $\tilde\lambda^2$, and then use these to derive ODEs for
$\E\det\tilde F$, $ \E(\det\tilde F)^2$, 
$\E |\tilde F|^2 $, and $ \E |\tilde F|^4$.
Finally, we use these ODEs to derive the precise
asymptotics for $\tilde\lambda\gg 1$.

\medskip

The outcome of our rigorous analysis allows for the following
informal interpretation: As $\varepsilon\ll 1$, 
the first r.~h.~s.~term in (\ref{wr36}) is dominant, that is,
$\tilde F$ is close to $\tilde F_{exp}$, where $\tilde F_{exp}$ is the 
stochastic exponential defined through
\begin{align}\label{ho34}
d\tilde F_{exp}=\tilde F_{exp}\nabla d\phi\;\;\mbox{and}\;\;
\tilde F_{exp}|_{\tilde\lambda^2=1}={\rm id}.
\end{align}
More precisely, $\tilde F_{exp}$ is the stochastic exponential of the martingale obtained from integrating $ \nabla d\phi$. 
Since by (\ref{wr01}) $\tilde\lambda^{-1}\nabla d\phi$ is the Helmholtz projection of $d\psi J$, 
and $\psi=\int_1^{ \cdot }d\psi$ is the Gaussian free field,
(see the second item in (\ref{ho36}) for a more direct relation between
$\tilde\lambda^{-1}\nabla d\phi$ and $d\psi$), 
we may assimilate $\tilde F_{exp}$ to a Gaussian multiplicative chaos. 

\medskip

In fact, in Fourier space, definition (\ref{wr01}) 
assumes the simple form of ${\mathcal F}d\phi$ 
$=i\frac{{\mathcal F}d\psi}{\tilde\lambda}\frac{(J k)^*}{|k|^2}$,
where the wave vector $k$ naturally is a co-tangent vector, and for a
cotangent vector $\xi$, $\xi^*$ denotes the associated tangent vector via the
inner product. The driving (Gaussian) martingale  
$\nabla\phi':=\int_1^{\cdot} \nabla d\phi$ in (\ref{ho34})
is characterized by its covariance tensor $C'$, 
which is a four-tensor; its Fourier transform is explicit:
\begin{align}\label{so14}
{\mathcal F}C'(k)=\frac{\varepsilon^2}{1-\varepsilon^2\ln|k|}I(L^{-1}<|k|\le1)
\underbrace{\frac{1}{|k|^6}((Jk)^*\otimes k)\otimes((Jk)^*\otimes k)}.
\end{align}
We remark that the underbraced part scales as for the Gaussian free field, that is, $|k|^{-2}$.
For very large scales, i.~e.~$\varepsilon^2\ln\frac{1}{|k|}\gg 1$, 
the prefactor is $\approx-\frac{1}{\ln|k|}$
and thus independent of $\varepsilon$; in this range,
the overall scaling is $(|k|^2\ln\frac{1}{|k|})^{-1}$. Since this expression is
just not integrable for $|k|\downarrow 0$, this logarithmic discount does
not prevent an IR-divergence as $L\uparrow\infty$.

\medskip

However, the tensor $\nabla d\phi$ has a special structure, as reflected by
its quadratic variation\footnote{which is a four-tensor} $\diamond$, 
see definition (\ref{wr51}). Since $\tr\nabla d\phi=0$, 
which is the first item in (\ref{ho36}), and in view of the property (\ref{detharm})
of $\diamond$, we have 
\begin{align}
d\det\tilde F_{exp}=0\quad\mbox{so that}\quad\det\tilde F_{ exp }=1,\label{ho35}
\end{align}
see Section \ref{s:proofs} for all details.
Likewise, the property (\ref{wr38d}) of $\diamond$ 
yields\footnote{see below for the standard  definition of $F^*$}
\begin{align}
d|\tilde F_{exp}|^2=
\frac{1}{2}|\tilde F_{exp}|^2\frac{d\tilde\lambda^2}{\tilde\lambda^2}
+2\tr\tilde F_{exp}^*\tilde F_{exp}\nabla d\phi.\label{ho32}
\end{align}
In a second stage this implies, appealing to the finer structure (\ref{wr45bis})
of $\diamond$,
\begin{align}\label{ho33}
d|\tilde F_{exp}|^4 = \big( \frac{3}{2}|\tilde F_{exp}|^4- 2 ( \det \tilde F_{ exp } )^2 \big) \frac{d\tilde\lambda^2}{\tilde\lambda^2} 
+\mbox{martingale},
\end{align}
where $\frac{3}{2}$ arises as $2\times\frac{1}{2}+2^2\times\frac{1}{8}$ with 
$\frac{1}{2}$ coming from (\ref{ho32}) and $\frac{1}{8}$
coming from (\ref{wr45bis}).
It is the integrating factor $\frac{1}{2\tilde\lambda^2}$ in (\ref{ho32}) 
that is responsible for the growth of order $ (\tilde\lambda^2)^\frac{1}{2}= \tilde\lambda$
in (\ref{mr05}); this was already clear from \cite{CMOW}.
Hence the factor $\frac{1}{2}$ in (\ref{ho32}) matters.
Intermittency arises from the fact that the factor $\frac{3}{2}$ in (\ref{ho33})
is (strictly) larger than twice the factor $\frac{1}{2}$ of (\ref{ho32}).
Hence the factor $\frac{3}{2}$
in (\ref{ho33}) matters (relative to the one in (\ref{ho32})).
In this sense, the tensorial stochastic exponential defined
through (\ref{ho33}) behaves as if
$\tilde F_{exp}$ and $\nabla d\phi$ were scalar. 
In fact, $\tilde F_{exp}$ is close to being rank-one 
in view of the (near-)degeneracy of $\tilde F_{exp}$ in form of $|\det\tilde F_{exp}|=1
\ll\tilde\lambda^2= \E|\tilde F_{exp}|^2$,
which in turn is a consequence of (\ref{ho35}) \& (\ref{ho32}).
Presently, our analysis is limited to monitoring quadratic and quartic
expressions in $\tilde F_{exp}(x)$ for the same point $x$ (which we suppress in our notation
because of stationarity). However, (\ref{ho34}) also allows us to identify evolution equations for multi-point covariance functions of $\tilde F_{exp}$; likewise, (\ref{wr04}) allows to characterize the evolution of increments of $\tilde\phi$, which we also expect to be intermittent.

\medskip

In the upcoming Lemma \ref{lem:evolution-f}, the endomorphism $F^t$ 
of tangent space denotes the transpose of the endomorphism $F$ of cotangent space
(transposition is a purely linear algebra operation), whereas the endomorphism $F^*$ 
on cotangent space denotes its adjoint
w.~r.~t.~the inner product $\cdot$.
It also features the adjugate\footnote{In coordinates $ F = F_j^i e_i \otimes e^j $ we have $ \adj F = F_2^2 e_1 \otimes e^1 + F_1^1 e_2 \otimes e^2 - F_1^2 e_1 \otimes e^2 - F_2^1 e_2 \otimes e^1 $.}
$\adj F:=(\det F)F^{-1}$, which in $d=2$ is linear in $F$.

\begin{lemma}[ODEs for moments of $ \tilde F $ and $ det\tilde F $]\label{lem:evolution-f}
We have 
\begin{align}
d\mathbb{E}\det \tilde F &= 0,\label{wr44} \\
d\mathbb{E}(\det \tilde F)^2&=\mathbb{E}
((\adj \tilde F)^t \otimes\tilde\phi)
\bullet((\adj \tilde F)^t \otimes\tilde\phi) 
\frac{ d \tilde\lambda^2 }{ L^2 \tilde\lambda^2 }, \label{wr45}\\
d\mathbb{E}| \tilde F|^2&=\big(\frac{1}{2}\mathbb{E}| \tilde F|^2+\frac{\mathbb{E}|\tilde\phi|^2}{2 L^2}\big) 
\frac{d\tilde\lambda^2}{\tilde\lambda^2},\label{wr43}\\
d\mathbb{E}| \tilde F|^4&=\big(\frac{3}{2}\mathbb{E}| \tilde F|^4-2\mathbb{E}(\det \tilde F)^2\big)\frac{d\tilde\lambda^2}{\tilde\lambda^2} \label{wr42} \\
& + \big( \mathbb{E} | \tilde\phi |^2 | \tilde F |^2 + 4 \mathbb{E} (\tilde F^{*t}\otimes\tilde\phi) 
\bullet (\tilde F^{*t}\otimes\tilde\phi) \big) 
\frac{d\tilde\lambda^2}{L^2 \tilde\lambda^2}, \nonumber
\end{align}
where 
$\bullet$ is a universal symmetric positive semi-definite bilinear form on 
the tensor space $(\mbox{endomorphisms of tangent space})\otimes(\mbox{tangent space})$.
\end{lemma}

Here, universal means that $\bullet$ is independent of any parameter
like $\varepsilon$ and $L$; $\bullet$ is defined in (\ref{wr80}), and characterized
as being universal in Appendix \ref{ss:var-nabla-nabla-phi}.
The identities of Lemma \ref{lem:evolution-f}
are much more explicit than needed for Theorem \ref{thm:proxy-intermittency}.
As we know from \cite{CMOW}, see Lemma \ref{lem:evolution-proxy-phi}, 
$\tilde\phi$ is of order $\frac{\varepsilon L}{\tilde\lambda}$;
hence in view of (\ref{wr45}) and (\ref{wr42}), 
$\det\tilde F$ is of lower order w.~r.~t.~$|\tilde F|^2$.
For the same reason,
the constants in front of the terms involving $\tilde\phi$ are irrelevant; 
so is the specific form of $\bullet$.
However, for this we need to capture 
$\tilde\phi=O(\frac{\varepsilon L}{\tilde\lambda})$ 
up to fourth moments, as already done in
\cite[(72)]{CMOW} and \cite[(30) \& (38)]{SuppCMOW} 
based on (\ref{wr04}) and once more It\^{o} calculus.
Again, we are more explicit here,
although the factors don't matter for the asymptotic scaling.

\begin{lemma}[Evolution of moments of $ \tilde \phi $]\label{lem:evolution-proxy-phi}
We have
\begin{align}
d\mathbb{E}|\tilde\phi|^2&=(\frac{1}{2}\mathbb{E}|\tilde\phi|^2+L^2)
\frac{d\tilde\lambda^2}{\tilde\lambda^2},\label{wr67}\\
d\mathbb{E}|\tilde\phi|^4&=(\frac{3}{2}\mathbb{E}|\tilde\phi|^4+4L^2\mathbb{E}|\tilde\phi|^2)
\frac{d\tilde\lambda^2}{\tilde\lambda^2}.\label{wr70}
\end{align}
\end{lemma}

Equipped with these evolution equations for
$ \E |\tilde\phi |^2 $ and $ \E |\tilde\phi |^4 $, we derive the 
precise asymptotics for these quantities, 
which already appeared in \cite[(69) \& (70)]{CMOW} and \cite[(31) \& (39)]{SuppCMOW}.

\begin{lemma}[Integration for $ \tilde \phi $]\label{lem:growth-proxy-phi}
We have
\begin{align}
\mathbb{E}|\tilde\phi|^2 & \approx \frac{ \varepsilon^2 }{ 2 } \frac{ L^2 }{ \tilde\lambda^2 }, \label{ode02a} \\
\mathbb{E}|\tilde\phi|^4 & \approx  \frac{ \varepsilon^4 }{ 2 } \frac{ L^4 }{ \tilde\lambda^4 } \label{ode03a}
\end{align}
as $L\gg 1$ and uniformly in $\varepsilon\le 1$.
\end{lemma}

Note that while (\ref{wr67}) \& (\ref{wr70}) determine the evolution
of $\E|\tilde\phi|^2$ and $\E|\tilde\phi|^4$, the ODEs of Lemma \ref{lem:evolution-f}
do not close on the level of
$\E\det\tilde F$, $\E(\det\tilde F)^2$, $\E|\tilde F|^2$, $\E|\tilde F|^4$.
However, both lemmas combined provide precise asymptotics for these quantities.
The rest of the paper is organized as follows. In order to prove Lemma \ref{lem:evolution-f} 
and Lemma \ref{lem:evolution-proxy-phi}, we identify a number of quadratic variations for 
the drivers $ d \phi $ and its derivatives. We introduce the necessary notation in 
Section \ref{sec:algebraic-prelim}. Therein we also collect the necessary identities, 
while the main computational work is deferred to Appendix 
\ref{ref:quadratic-variation}. Proofs of the above results are given in Section 
\ref{section:evolution} (where we derive the differential equations for moments), 
in Section \ref{section:proofs} (where identify the asymptotic), 
and Section \ref{section:proof-mr} (where we 
analyze a backward Kolmogorov equation for $|\tilde F|^2$ as needed for
(\ref{ao13})).


\section{Proofs}\label{s:proofs}

\subsection{Identification of quadratic variations}\label{sec:algebraic-prelim}

In order to establish Lemmas \ref{lem:evolution-f} and 
\ref{lem:evolution-proxy-phi}, 
in view of (\ref{wr04}) and (\ref{wr36}) and It\^{o} calculus,
we need to identify the quadratic variations of the drivers $ \nabla d \psi $, $ d\phi $, 
$ \nabla d \phi $, and $\nabla^2 d\phi$ by relating them to
the one of $d\psi$, see (\ref{in09}). 
The underlying principle that a derivative gains a factor of $ L^{-1} $ is best seen for $ \psi $,~i.~e.~
\begin{equation}\label{qvarpsi}
[ \nabla d \psi \cdot \nabla d \psi ]
= \frac{ [ d \psi \, d \psi ] }{ L^2 }.
\end{equation}
Indeed, on the level of covariance functions and their Fourier-transforms we have
$$
c_{ \nabla \psi_{ L_+ } -  \nabla \psi_{ L } } = - \nabla^2 c_{ \psi_{ L_+ } - \psi_{ L } } 
\quad { \rm so ~ that } \quad
\tr \mathcal{F} c_{ \nabla \psi_{ L_+ } -  \nabla \psi_{ L } } = | k |^2 \mathcal{F} c_{ \psi_{ L_+ } - \psi_{ L } } 
~ { \rm distributionally }.
$$
From the latter identity, it follows that
$$ 
\frac{1}{L_+^2} \E ( \psi_{ L_+ } - \psi_{ L } )^2
\leq  \E | \nabla \psi_{ L_+ } - \nabla \psi_{ L } |^2
\leq  \frac{1}{L^2} \E ( \psi_{ L_+ } - \psi_{ L } )^2
$$
for $ L+ \ge L $. Using the characterization of $ [ d \psi \, d \psi ] $ and $ [ \nabla d \psi \cdot \nabla d \psi ] $ in Footnote \ref{footnote:quadvar}, we have
$$
\begin{aligned}
\int_{\ell_1}^{\ell_2} [ \nabla d \psi \cdot \nabla d \psi ]
&= \lim_{ { \rm fineness } ( \{ L_k \} ) \rightarrow 0 }  \sum_{ k } \E | \nabla \psi_{ L_{k} } - \nabla \psi_{ L_{k-1} } |^2 \\
&= \lim_{ { \rm fineness } ( \{ L_k \} ) \rightarrow 0 }  \sum_{ k } \E  \frac{1}{L_k^2}( \psi_{ L_{k} } - \psi_{ L_{k-1} } )^2 \\
&= \int_{ \ell_1 }^{ \ell_2 } \frac{ [ d \psi \, d \psi ] }{ L^2 },
\end{aligned}
$$
which establishes (\ref{qvarpsi}).

\medskip

We will now collect the necessary formulas for $ d \phi $, $ \nabla d \phi $ and 
$ \nabla^2 d \phi $. In view of (\ref{wr01}), the most
natural object out of these three is 
$ \nabla d \phi $ as it is directly related to $ d \psi $. Since we
move most computations to the appendix, 
we however start the presentation with $ d \phi $.

\begin{lemma}[Quadratic variation of $ d\phi $]\label{lem:quad-var-phi}
Given a cotangent vector $ \xi $, it holds\footnote{We use the symbol $ . $ to denote a 
canonical pairing between primal and dual space, 
here $ \xi . \dot x = \xi_i \dot x^i $, to be distinguished from the inner
product $\dot x\cdot \dot y=\sum_{i=1}^2\dot x^i\dot y^i$.}
\begin{align}\label{wr75}
[ \xi.  d\phi \, \xi. d\phi ] =
| \xi |^2 \frac{L^2 d\tilde\lambda^2}{2\tilde\lambda^2}.
\end{align}
As a special case of this identity, we have
\begin{align}\label{wr75bis} 
[d\phi \cdot d\phi] = L^2 \frac{d \tilde\lambda^2}{\tilde\lambda^2}.
\end{align}
\end{lemma}

A proof of Lemma \ref{lem:quad-var-phi} is given after the next lemma. 
Identifying the other variations is more involved because the processes are tensor
valued (fields). We thus consider general linear forms $G$ on these tensor spaces,
which act as observables. All quadratic variations
come as positive semi-definite bilinear forms on the linear space of $G$'s; 
like all quadratic variations of processes with independent increments, 
these forms are deterministic. Since the processes arise from evaluating
a spatially stationary field in $x$, these forms are independent of $x$,
see Footnote \ref{footnote:quadvar}.
As a consequence, since $\nabla d\phi$ is 
the increment of such a process with values in the space
of endomorphisms of cotangent space,
the corresponding quadratic variation is such a deterministic
form on the dual space, which is canonically isomorphic\footnote{
via the pairing $(\xi\otimes\dot x).F$
$:=(F\xi).\dot x$ for a tangent vector $\dot x$, a cotangent vector $\xi$, and 
an endomorphism $F$ of cotangent space, and where the $.$ on the r.~h.~s.~denotes 
the natural pairing between cotangent and tangent vectors} to
\begin{align*}
\mbox{endomorphisms of tangent space}\cong
(\mbox{cotangent space}) \otimes (\mbox{tangent space}).
\end{align*}
Likewise, 
the quadratic variation of $\nabla^2d\phi$ is a deterministic and $x$-independent
positive semi-definite bilinear form on 
$(\mbox{cotangent space}) \otimes (\mbox{tangent space}) \otimes (\mbox{tangent space})$.
Let us summarize this in the following definition.

\begin{definition}\label{defn:quad-var}
\begin{enumerate}[\hspace{0cm}i)]
\item For $G$ $\in(\mbox{cotangent space}) \otimes (\mbox{tangent space})$, we write
\begin{align}\label{wr51}
G \diamond G \frac{d\tilde\lambda^2}{\tilde\lambda^2} \coloneqq [G.\nabla d\phi \, G.\nabla d\phi ],
\end{align}
where $ G . \nabla d \phi $ denotes the natural pairing
which componentwise is given by $ G_j^i \partial_i d \phi^j $.

\vspace{5pt}

\item For $ \textit{\textbf{G}} \in (\mbox{cotangent space}) \otimes(\mbox{tangent space}) \otimes(\mbox{tangent space}) $, we let
\begin{align}\label{wr80}
\textit{\textbf{G}} \bullet \textit{\textbf{G}} \frac{d\tilde\lambda^2}{L^2\tilde \lambda^2}
\coloneqq [ \textit{\textbf{G}}.\nabla^2d\phi \, \textit{\textbf{G}}.\nabla^2d\phi ],
\end{align}
where $ \textit{\textbf{G}}. \nabla^2d\phi $ denotes the natural pairing
which componentwise is given by  $ \textit{\textbf{G}}^{ij}_{k} \partial_i\partial_j d \phi^k $.
\end{enumerate}
\end{definition}

Note that via polarization, (\ref{wr51}) and (\ref{wr80}) indeed define 
symmetric bilinear forms.  
There is no need to track the mixed covariation of zeroth, first, 
and second derivative:

\begin{remark}\label{rmk:vanishing-covariation}
For every $ \xi \in \mbox{cotangent space} $, 
$G$ $\in(\mbox{cotangent space}) \otimes (\mbox{tangent space})$ and
$ \textit{\textbf{G}} \in (\mbox{cotangent space}) \otimes(\mbox{tangent space}) \otimes(\mbox{tangent space}) $
the mixed covariations $ [ \xi . d \phi \, G.\nabla d\phi] $ and 
$[G.\nabla d\phi\,\textit{\textbf{G}}.\nabla^2 d\phi]$
vanish because of parity. By this we mean that via \eqref{wr01}, 
the process $\nabla d\phi$ inherits the evenness (in law) of $\psi$ under point reflections 
in space; which implies that $d\phi$ and $ \nabla \partial_i d \phi $ are odd.
Since their co-variation is deterministic by independence of increments, it must vanish.
\end{remark}

In Subsections \ref{ss:var-nabla-phi} and \ref{ss:var-nabla-nabla-phi} 
we identify natural bases that diagonalize these forms, based on their invariance under
the action of rotations:
In (\ref{wr82}), we introduce such a basis
$ (E^n)_{ n = 1,2,3,4 } $ of $(\mbox{cotangent space})\otimes(\mbox{tangent space})$.
In (\ref{wr85}) we define a suitable basis $ ( \textit{\textbf{E}}^n )_{ n = 1, \hdots, 6 } $ 
of $(\mbox{cotangent space})\otimes(\mbox{tangent space})
\otimes_{sym}(\mbox{tangent space})$, where the subscript $sym$ indicates the symmetry in
the last two factors, to which we may restrict in view 
of $\nabla^2d\phi$ in view of the symmetry of second derivatives.
In the rest of this section we record some contractions of these forms for later use.

\medskip

We start with the contractions of $\diamond$ that are relevant for
computing $d|\tilde F|^2$ and $d|\tilde F|^4$.  For $d|\tilde F|^2$, we
represent the (squared) Frobenius norm $|F|^2= \tr F^*F=\tr FF^*$ 
of an endomorphism as the following sum of quadratic form
\begin{align}\label{ho38}
|F|^2=\sum_{i,j = 1}^2 (e^i\otimes e_j.F)(e^i\otimes e_j.F).
\end{align}
In view of the first r.~h.~s.~term of (\ref{wr36}), we apply (\ref{ho38}) to
$F=\tilde F\nabla d\phi$, so that we actually need to monitor the more general quadratic form
on the space of endomorphisms of cotangent space
\begin{align}\label{ho38bis}
F\mapsto \sum_{i,j = 1}^2 (Ge^i\otimes e_j.F)(Ge^i\otimes e_j.F),
\end{align}
where the endomorphism $G$ on tangent space is the placeholder for $\tilde F^t$;
the outcome is given in (\ref{wr38d}).
We now turn to $d|\tilde F|^4$;
in order to identify the contribution of the martingale term in (\ref{ho32}) to (\ref{ho33}),
we need to characterize the quadratic variation of $G.\nabla d\phi$ 
for a symmetric endomorphism $G$ on tangent space, which plays the role of $(F^*F)^t$;
see (\ref{wr45bis}) for the outcome.
Two contractions of $\diamond$ are relevant for $d|\tilde\phi|^2$ and $d|\phi|^4$:
In order to identify the contribution of the second r.~h.~s.~term of (\ref{wr04})
to $|\tilde\phi|^2$, we need to characterize the quadratic variation of 
$\dot x\cdot\nabla d\phi^i$ $=(e^i\otimes\dot x).\nabla d\phi$,
for a tangent vector $\dot x$ playing the role of $\tilde\phi$, see (\ref{wr65bis})
for the result.
For $|\tilde\phi|^4$ we need (\ref{wr72}).
Finally, the contraction (\ref{wr65}) is needed to pass from the quadratic variation of 
$\nabla d\phi$ to the one of $d\phi$. The factor $\frac{1}{2}$ in (\ref{wr38d}), (\ref{wr65bis}),
and (\ref{wr65}) has the same origin. The factors $\frac{1}{2}$ in (\ref{wr38d}) and
$\frac{1}{8}$ in (\ref{wr45bis}) are crucial.

\begin{lemma}[Quadratic variation of $\nabla d\phi$ for $|\tilde F|^2$,
$|\tilde F|^4$, $|\tilde\phi|^2$, $|\tilde\phi|^4$]
\label{lem:quad-var-nabla-phi}
For an endomorphism $G$ of tangent space we have in terms of the trace-free part
$G_{sym}$ $- {\textstyle\frac{1}{2}} (\tr G_{sym}){\rm id}$
of its symmetric part $G_{sym}=\frac{1}{2}(G+G^*)$ and its skew-symmetric part
$G_{skew}:=\frac{1}{2}(G-G^*)$
\begin{align}\label{wr45base}
G \diamond G= \frac{1}{4}|G_{sym}- {\textstyle\frac{1}{2}} (\tr G_{sym}){\rm id}|^2 +\frac{1}{2} |G_{skew}|^2,
\end{align}
which we use in form of
\begin{equation}\label{wr45bis}
G\diamond G = \frac{1}{8}(\tr G)^2-\frac{1}{2}\det G\quad\mbox{provided}\;G\;\mbox{is symmetric}.
\end{equation}
As a consequence of (\ref{wr45base}) we have\footnote{
The product $Ge^i\otimes e_j$ denotes the composition of endomorphisms (on tangent space),
that is, $G e^i\otimes e_j=e^i\otimes Ge_j$.
In general, we often suppress brackets based on the convention that
$\otimes$ binds more tightly than composition, pairing, and inner products like
$\cdot$, $\diamond$, or $\bullet$.}
%
\begin{equation}\label{wr38d}
\sum_{i,j=1}^2 G e^i \otimes e_j \diamond G e^i\otimes e_j = \frac{1}{2}| G |^2
\end{equation}
and for any cotangent vector $\xi$ and any tangent vector $\dot x$
\begin{align}
\sum_{i=1}^2e^i \otimes \dot x \diamond e^i \otimes \dot x
&= \frac{1}{2} | \dot x  |^2,\label{wr65bis}\\
\dot x^* \otimes \dot x \diamond \dot x^* \otimes \dot x
&= \frac{1}{8} |\dot x|^4,\label{wr72}\\
\sum_{i=1}^2 \xi \otimes e_i \diamond \xi \otimes e_i
&= \frac{1}{2} |\xi|^2,\label{wr65}
\end{align}
where $\dot x^*$ denotes the cotangent vector associated
to $\dot x$ by the 
inner product, in coordinates $ (\dot x^*)_i = \dot x^i $. 
\end{lemma}

\begin{proof}[Proof of Lemma \ref{lem:quad-var-nabla-phi}.]$\mbox{}$ 
\medskip


\textit{Argument for (\ref{wr45base})}.
According to Appendix \ref{ss:qv1}, both $\diamond$ and the standard inner product given by
\begin{align}\label{ho40}
G:G'= \tr G^* G'
\end{align}
are diagonal w.~r.~t.~the decomposition of endomorphism space
\begin{align}\label{wr50}
\mbox{trace-free symmetric}\oplus\mbox{skew}\oplus\mbox{isotropic}.
\end{align}
More precisely, it follows from (\ref{ho41}) and (\ref{wr52}) \& (\ref{wr53}) that 
$\diamond$ is $\frac{1}{4}$ times $:$ on the first component, 
whereas the factor is $\frac{1}{2}$ on the second, and $0$ on the last.
Since $G_{sym}-\frac{1}{2}(\tr G_{sym}){\rm id}$ and $G_{skew}$ are the $:$-orthogonal
projections on the first and second component, respectively, (\ref{wr45base}) follows.

\medskip

\textit{Argument for (\ref{wr45bis})}. Passing from
(\ref{wr45base}) to (\ref{wr45bis}) reduces to the identity for symmetric matrices $G$
\begin{align*}
2|G-\frac{1}{2}(\tr G){\rm id}|^2=(\tr G)^2-4\det G,
\end{align*}
which by definition of $|\cdot|^2$ is equivalent to
$|G|^2=(\tr G)^2-2\det G$, which in terms of the eigenvalues amounts to
$\lambda_1^2+\lambda_2^2=(\lambda_1+\lambda_2)^2-2\lambda_1\lambda_2$ and thus holds true.

\medskip

\textit{Argument for (\ref{wr65bis})}.
%
%
In view of the invariance of $\diamond$ under the action (\ref{ho28}) 
of rotations we have in particular for the rotation $J$ by $\frac{\pi}{2}$
\begin{align*}
e^i\otimes\dot x\diamond e^i\otimes\dot x
=J.e^i\otimes\dot x\diamond J.e^i\otimes\dot x
=J^{-t}e^i\otimes J\dot x \diamond J^{-t}e^i\otimes J\dot x.
\end{align*}
Since $\{\pm J^{-t}e^i\}_{i=1,2}$ is a permutation of $\{\pm e^i\}_{i=1,2}$,
the quadratic form in $\dot x$ on the
l.~h.~s.~of (\ref{wr65bis}) is invariant under $J$, and thus isotropic.
Hence it is enough to establish (\ref{wr65bis}) for $\dot x=e_j$, that is
\begin{align}\label{ho45}
\sum_{i=1}^2e^i\otimes e_j\diamond e^i\otimes e_j=\frac{1}{2}.
\end{align}

\medskip

We now derive (\ref{ho45}) from (\ref{wr45base}). For this purpose, we rewrite (\ref{wr45base})
as
\begin{align}\label{ho46}
G\diamond G=\frac{1}{4}|G|^2-\frac{1}{8}\big((\tr G)^2-(\tr J^*G)^2\big),
\end{align}
which follows from the $:$-related identities
\begin{align*}
|G_{sym}-\frac{1}{2}(\tr G_{sym}){\rm id}|^2 +|G_{skew}|^2
&=|G-\frac{1}{2}(\tr G){\rm id}|^2
=|G|^2-\frac{1}{2}(\tr G)^2,\\
|G_{skew}|^2&=|\frac{1}{2}(\tr J^*G)J|^2=\frac{1}{2}(\tr J^*G)^2.
\end{align*}
It remains to note that $G^i:=e^i\otimes e_j$ satisfies $|G^i|^2=1$ 
and thus $\sum_{i=1}^2|G^i|^2=2$.
It also satisfies $(\tr G^i)^2=\delta^i_j$ while $(\tr J^* G^i)^2=1-\delta^i_j$ and
thus $\sum_{i=1}^2((\tr G^i)^2-(\tr J^*G^i)^2)=0$.

\medskip

\textit{Argument for (\ref{wr72})}. This follows from
(\ref{ho46}) for $G=\dot x^*\otimes\dot x$ which satisfies $|G|^2=|\dot x|^4$,
$\tr G=|\dot x|^2$ and $\tr J^*G=0$.

\medskip

\textit{Argument for (\ref{wr65})}. We read off 
(\ref{wr45base}) that $\diamond$ is invariant under $G\mapsto G^*$.
Since $(e^i\otimes\dot x)^*=\dot x^*\otimes e_i$, (\ref{wr65}) thus follows from  
(\ref{wr65bis}) for $\dot x^*=\xi$.

\medskip

\textit{Argument for (\ref{wr38d})}.
Noting that
$Ge^i\otimes e_j=e^i\otimes Ge_j$ 
we see that the l.~h.~s.~of (\ref{wr38d}) is given by
\begin{align*}
\sum_{j=1}^2\sum_{i=1}^2e^i\otimes Ge_j\diamond e^i\otimes Ge_j
\stackrel{(\ref{wr65bis})}{=}\frac{1}{2}\sum_{j=1}^2|Ge_j|^2=\frac{1}{2}|G|^2,
\end{align*}
as desired.
\end{proof}

We are now ready to characterize also the quadratic variation of $ d \phi $.

\begin{proof}[Proof of Lemma \ref{lem:quad-var-phi}.]
For the purpose of this argument, 
we introduce the tangent-vector valued
process\footnote{which is not $L\mapsto\phi_L$, see Footnote \ref{footnote:dphi}  }
$L\mapsto\phi_L'$ that is obtained from integrating $d\phi$.
We start as for (\ref{qvarpsi}) and
observe that on the level of covariance functions we have
$$
- \nabla^2 c_{ \xi . (\phi_{L_+}'-\phi_{L}') } 
= c_{ \nabla(\xi . (\phi_{L_+}' - \phi_{L}') ) }
\quad\mbox{and thus}\quad
- \triangle c_{ \xi . (\phi_{L_+}'-\phi_{L}')} 
= \sum_{i=1}^2c_{\xi . \partial_i(\phi_{L_+}' - \phi_{L}') }.
$$
As for $\nabla\tilde\phi_L$,
we interpret $\nabla \phi_L' $ as having values in the space of endomorphisms of cotangent space
$\cong(\mbox{tangent space})\otimes(\mbox{cotangent space})$, which implies
$\xi . \partial_i\phi_L' =( \xi \otimes e_i).\nabla\phi_L' $.
Hence on the Fourier level, the above translates to
$$
| k |^2 \mathcal{F} c_{ \xi . (\phi_{L_+}' - \phi_{L}' ) }
= \sum_{i=1}^2 \mathcal{F} c_{ ( \xi \otimes e_i ) .(\nabla\phi_{L_+}' - \nabla\phi_{L}' ) }.
$$
Therefore, the same argument that lead to (\ref{qvarpsi}) yields
\begin{align*}
L^2 [ \xi . d \phi \, \xi . d \phi ]
&= \sum_{i=1}^2 [ ( \xi \otimes e_i ) . \nabla d \phi 
\, ( \xi \otimes e_i) . \nabla d \phi ]\\
& \stackrel{(\ref{wr51})}{=} \sum_{i=1}^2
\xi \otimes e_i \diamond \xi \otimes e_i
\frac{d \tilde\lambda^2 }{ \tilde\lambda^2 }
\stackrel{(\ref{wr65})}{=}
|\xi|^2 \frac{d \tilde\lambda^2 }{ 2\tilde\lambda^2 },
\end{align*}
as desired.
\end{proof}

The quadratic variation $\bullet$ arises from the second r.~h.~s.~term of the It\^{o}
SDE (\ref{wr36}). The following contraction comes up when computing $d|\tilde F|^2$,
and is easily made explicit.
 
\begin{lemma}[Quadratic variation of $\nabla^2d\phi$ for $d|\tilde F|^2$]
\label{lem:quad-var-nabla-nabla-phi}
For any tangent vector $\dot x$ we have
\begin{equation}\label{dc03}
\sum_{i,j=1}^2 e^i\otimes e_j\otimes\dot x\bullet e^i\otimes e_j\otimes\dot x
= \frac{1}{2}|\dot x|^2.
\end{equation}
\end{lemma}

\begin{proof}
We note that the l.~h.~s.~of (\ref{dc03}) defines a (non-negative) quadratic form on tangent space.
Since $\bullet$ is invariant under the action (\ref{wr78}) of rotations, 
and in particular under $J$, we have
\begin{align*}
J.e^i\otimes e_j\otimes\dot x\bullet J.e^i\otimes e_j\otimes\dot x
=J^{-t}e^i\otimes Je_j\otimes J\dot x\bullet J^{-t}e^i\otimes Je_j\otimes J\dot x.
\end{align*}
Since $\{(J^{-t}e^i,Je_j)\}_{i=1,2}$ is a permutation of $\{(e^i,e_j)\}_{i=1,2}$ modulo the sign,
we learn that this quadratic form is invariant under $J$.
Hence it must be isotropic, so that it is enough to show
\begin{align}\label{ho13}
\sum_{i,j=1}^2 e^i\otimes e_j\otimes e_1\bullet e^i\otimes e_j\otimes e_1
=\frac{1}{2}.
\end{align}

\medskip

To this purpose, 
we use (\ref{wr86}) and (\ref{wr86bis}) to express $e^i\otimes e_j\otimes e_1$
in terms of the $\textit{\textbf{E}}^{1},\hdots,\textit{\textbf{E}}^{6}$:
\begin{align*}
4\,e^1 \otimes e_1 \otimes e_1 &= 
- \textit{\textbf{E}}^1 + \textit{\textbf{E}}^3 + \textit{\textbf{E}}^5,\\
4\,e^1 \otimes e_2 \otimes e_1&=_{sym}
\textit{\textbf{E}}^2 - \textit{\textbf{E}}^4 + \textit{\textbf{E}}^6,\\
4\,e^2 \otimes e_1 \otimes e_1&=
\textit{\textbf{E}}^2 + 3\, \textit{\textbf{E}}^4- \textit{\textbf{E}}^6,\\
4\,e^2 \otimes e_2 \otimes e_1&=_{sym}
\textit{\textbf{E}}^1 - \textit{\textbf{E}}^3 + \textit{\textbf{E}}^5,
\end{align*}
where the subscript sym indicates that the identity holds only modulo
symmetrization of the last two factors in the tensor product
$(\mbox{cotangent space})\otimes(\mbox{tangent space})\otimes(\mbox{tangent space})$.
In view of (\ref{wr83}) we have 
\begin{align*}
2(-\textit{\textbf{E}}^1+\textit{\textbf{E}}^3+\textit{\textbf{E}}^5)\bullet
(-\textit{\textbf{E}}^1+\textit{\textbf{E}}^3+\textit{\textbf{E}}^5)&=1+1=2,\\
2( \textit{\textbf{E}}^2-\textit{\textbf{E}}^4+\textit{\textbf{E}}^6)\bullet
( \textit{\textbf{E}}^2-\textit{\textbf{E}}^4+\textit{\textbf{E}}^6)&=1+1=2,\\
2( \textit{\textbf{E}}^2 + 3\textit{\textbf{E}}^4- \textit{\textbf{E}}^6)\bullet
 (\textit{\textbf{E}}^2 + 3\textit{\textbf{E}}^4- \textit{\textbf{E}}^6)&=1+3^2=10,\\
2(\textit{\textbf{E}}^1 - \textit{\textbf{E}}^3 + \textit{\textbf{E}}^5)\bullet
(\textit{\textbf{E}}^1 - \textit{\textbf{E}}^3 + \textit{\textbf{E}}^5)&=1+1=2,
\end{align*}
so that the sum is $16$, which yields (\ref{ho13})
\end{proof}

Finally, we turn to the contributions of both $\diamond$ and $\bullet$ 
to $d\det\tilde F$.
Note that in $d=2$, the determinate $\det$ is a bilinear form on the space of
endomorphisms $F$ of cotangent space, as it can be written as 
\begin{align}\label{ho31}
\det F=F^1_1 F^2_2-F_1^2F_2^1=(e^1\otimes e_1.F)(e^2\otimes e_2.F)
-(e^1\otimes e_2.F)(e^2\otimes e_1.F).
\end{align}
The next lemma implies in particular that the contraction of $\det$ with $\diamond$, 
which is a canonical pairing since they are bilinear forms on dual spaces, 
and important since it arises as It\^{o} correction, vanishes.
In other words, $\det $ is a harmonic bilinear form w.~r.~t.~to the
Euclidean structure defined by $ \diamond $. Since via It\^{o}
calculus, the representation (\ref{ho31}) applied to $F=\tilde F$ will give rise to
(\ref{ho31}) with $F$ replaced by the composition $\tilde F \nabla d\phi$, 
we need (\ref{detharm}) with $G=\tilde F^t$.
 
\begin{lemma}[Quadratic variation of $\nabla d\phi$ and $\nabla^2d\phi$ for $d\det\tilde F$]\label{lem:detharm}
For any endomorphism $G$ of tangent space and any tangent vector $\dot x$ we have
\begin{align}
Ge^1\otimes e_1 \diamond Ge^2\otimes e_2
- Ge^1\otimes e_2 \diamond Ge^2\otimes e_1 = 0,\label{detharm}\\
e^1\otimes e_1\otimes\dot x \bullet e^2\otimes e_2\otimes\dot x
- e^1\otimes e_2\otimes\dot x \bullet e^2\otimes e_1\otimes\dot x =0.\label{detharm02}
\end{align}
\end{lemma}

\begin{proof}[Proof of Lemma \ref{lem:detharm}.]$\mbox{}$ 
\medskip

\textit{Argument for (\ref{detharm}).}
Note that the l.~h.~s.~of (\ref{detharm}), which we rewrite in terms of $G_i:=Ge_i$ as
\begin{align*}
e^1\otimes G_1 \diamond e^2\otimes G_2- e^1\otimes G_2 \diamond e^2\otimes G_1,
\end{align*}
is a skew-symmetric bilinear form in the pair of tangent vectors $(G_1,G_2)$,
and thus a scalar multiple of $\det G$. Hence it is enough to check (\ref{detharm})
for $G={\rm id}$:
\begin{align}\label{ho47}
e^1\otimes e_1 \diamond e^2\otimes e_2=e^1\otimes e_2 \diamond e^2\otimes e_1.
\end{align}

\medskip

In order to establish (\ref{ho47}), we appeal to the polarized version of (\ref{ho46}),
that is
\begin{align}\label{ho48}
G\diamond G'=\frac{1}{4}G:G'-\frac{1}{8}\big((\tr G)(\tr G')-(\tr J^*G)(\tr J^*G')\big),
\end{align}
which we use for $(G,G')=(e^1\otimes e_1,e^2\otimes e_2)$ and
$(G,G')=(e^1\otimes e_2,e^2\otimes e_1)$. For both couples, the first r.~h.~s.~term
of (\ref{ho48}) vanishes. For the first couple, $(\tr G)(\tr G')=1\times 1$ while
$(\tr J^*G)(\tr J^* G')=0\times 0$. Since for the second couple
$(J^*G,J^*G')=(e^1\otimes J^*e_2,e^2\otimes J^*e_1)
=(e^1\otimes e_1,-e^2\otimes e_2)$, we have
$(\tr G)(\tr G')=0\times 0$ and $(\tr J^*G)(\tr J^* G')=1\times(-1)$. Hence both
sides of (\ref{ho47}) are identical.

\medskip

\textit{Argument for (\ref{detharm02}).} 
By linearity, we need to show
\begin{align}
e^1\otimes e_1\otimes e_i \bullet e^2\otimes e_2\otimes e_j
= e^1\otimes e_2\otimes e_i \bullet e^2\otimes e_1\otimes e_j.\label{ho05}
\end{align}
For $i\not= j$, both sides of (\ref{ho05}) vanish: 
Indeed, consider the $\bullet$-product on either side, and note that
in one factor, the index $2$ appears an even number of times, 
while in the other, it appears odd times. Hence by (\ref{wr86}) and modulo symmetrizing
in the last two factors
one factor is in the span of $\{\textit{\textbf{E}}^1,\textit{\textbf{E}}^3,\textit{\textbf{E}}^5\}$, 
while by (\ref{wr86bis}) the other
is in the one of
$\{\textit{\textbf{E}}^2,\textit{\textbf{E}}^4,\textit{\textbf{E}}^6\}$. Hence by (\ref{wr83}),
the $\bullet$-product vanishes.

\medskip

For $i=j$ in (\ref{ho05}), we first note that by the invariance of $\bullet$ 
under the action (\ref{wr78}) of rotations, and by the symmetry of $\bullet$, (\ref{ho05}) reduces to 
\begin{align*}
e^1\otimes e_1\otimes e_1\bullet e^2\otimes e_2\otimes e_1
= e^2\otimes e_1\otimes e_2 \bullet e^1\otimes e_2\otimes e_2.
\end{align*}
By (\ref{wr86}) and dividing by $16$, this assumes the form
\begin{align*}
(-\textit{\textbf{E}}^1+\textit{\textbf{E}}^3+\textit{\textbf{E}}^5)\bullet (\textit{\textbf{E}}^1-\textit{\textbf{E}}^3+\textit{\textbf{E}}^5)
=(\textit{\textbf{E}}^1-\textit{\textbf{E}}^3+\textit{\textbf{E}}^5)\bullet (\textit{\textbf{E}}^1+3\textit{\textbf{E}}^3-\textit{\textbf{E}}^5).
\end{align*}
It follows from (\ref{wr83}) that both sides take the value $-1$.
\end{proof}


\subsection{Identification of It\^{o} SDEs}\label{section:evolution} We now give the 
proofs of Lemma \ref{lem:evolution-proxy-phi} and Lemma \ref{lem:evolution-f},
which state ODEs for (first and second) moments of 
$|\tilde\phi|^2$, $\det\tilde F$, and $|\tilde F|^2$.
In fact, we shall derive It\^{o} SDEs for these quantities themselves and their squares
(in case of the latter without keeping track of the martingale term),
so that
the lemmas follow after taking the expectation.

\begin{proof}[Proof of Lemma \ref{lem:evolution-proxy-phi}.]$\mbox{}$ 
\medskip

\textit{Identification of $d|\tilde\phi|^2$.} We claim that
\begin{align}\label{wr66}
d|\tilde\phi|^2 = 2 \tilde G.\nabla d\phi+2\tilde\phi\cdot d\phi
+(\frac{1}{2}|\tilde\phi|^2+L^2)\frac{d\tilde\lambda^2}{\tilde\lambda^2},
\end{align}
where $ \tilde G $ is the endomorphism of tangent space 
(of Gram type and thus symmetric)
\begin{align}\label{wr69}
\tilde G:=\tilde\phi^* \otimes\tilde\phi,
\quad\mbox{in coordinates}\;\;\tilde G^j_i = \tilde\phi^i\tilde\phi^j.
\end{align}
Note that since the two first r.~h.~s.~terms of (\ref{wr66}) are of vanishing expectation,
(\ref{wr66}) implies (\ref{wr67}).

\medskip

Applying Itô's formula to
(\ref{wr04}) and appealing to Remark \ref{rmk:vanishing-covariation} we obtain
\begin{eqnarray*}
d | \tilde\phi |^2
&=& 2 \tilde\phi \cdot d\tilde\phi + [ d \tilde\phi \cdot d \tilde\phi ] \\
&\stackrel{(\ref{wr04})}{=}&
2 \tilde\phi \cdot d\phi + 2 \tilde\phi^i \tilde\phi^j \partial_i d \phi^j+ 
[ d\phi \cdot d\phi ]+ \tilde\phi^i \tilde\phi^j 
[ \partial_i d\phi \cdot \partial_j d\phi ].
\end{eqnarray*}
Since $\partial_id\phi^k=(e^k\otimes e_i).\nabla d\phi$, we obtain by definitions
(\ref{wr51}) and (\ref{wr69})
\begin{eqnarray*}
d | \tilde\phi |^2&=&
2 \tilde\phi \cdot d\phi+2\tilde G . \nabla d \phi +
[ d\phi \cdot d\phi ]+ \sum_{k=1}^2 e^k\otimes\tilde\phi \diamond e^k\otimes\tilde\phi 
\frac{ d \tilde\lambda^2 }{ \tilde\lambda^2 }.
\end{eqnarray*}
It remains to insert (\ref{wr75bis}) and (\ref{wr65bis}) to obtain (\ref{wr66}).

\medskip

\textit{Identification of $ d \E | \tilde\phi |^4 $.}
Applying It\^{o}'s formula to (\ref{wr66}), 
appealing to Remark \ref{rmk:vanishing-covariation},
and inserting definition (\ref{wr51}) yields
\begin{align*}
d|\tilde\phi|^4=
(2 \tilde G) \diamond (2 \tilde G)\frac{ d\tilde\lambda^2}{\tilde\lambda^2}
+[2\tilde\phi\cdot d\phi \, 2\tilde\phi\cdot d\phi]
+2|\tilde\phi|^2
(\frac{1}{2}|\tilde\phi|^2 + L^2) 
\frac{d \tilde\lambda^2}{\tilde\lambda^2}+\mbox{martingale}.
\end{align*}
Inserting (\ref{wr75}) and (\ref{wr72}) 
(recall the definition $G$ in (\ref{wr69})) yields
\begin{align*}
d|\tilde\phi|^4
&=\frac{1}{2}|\tilde\phi|^4\frac{ d\tilde\lambda^2}{\tilde\lambda^2}
+2|\tilde\phi|^2L^2\frac{ d\tilde\lambda^2}{\tilde\lambda^2}
+2|\tilde\phi|^2(\frac{1}{2}|\tilde\phi|^2 + L^2)
\frac{d \tilde\lambda^2}{\tilde\lambda^2}+\mbox{martingale}\nonumber\\
&=( \frac{3}{2} |\tilde\phi|^4 + 4 L^2 | \tilde\phi |^2 ) \frac{d \tilde\lambda^2}{\tilde\lambda^2}
+ \mbox{martingale},
\end{align*}
which indeed implies the claim.
\end{proof}


\begin{proof}[Proof of Lemma \ref{lem:evolution-f}.]$\mbox{}$ 
\medskip

\textit{Identification of $d \det \tilde F $.} We claim that
\begin{align}\label{wr57}
d\det\tilde F 
= \tilde\phi^i(\adj \tilde F)^t. \nabla\partial_i d \phi.
\end{align}
In particular, it follows from (\ref{wr57}) that $\det F$ is a martingale,
which implies (\ref{wr44}).
The adjugate $\adj F=(\det F) F^{-1}$,
an endomorphism of tangent space, naturally arises in (\ref{wr57}) 
as the differential of the determinate:
If $F$ depends on some parameter $s$, we have
\begin{align}\label{ho03}
\frac{d}{ds}\det F=(\det F)\tr F^{-1}\frac{d}{ds} F=(\adj F)^t.\frac{d}{ds} F.
\end{align}
Using (\ref{ho03}), and appealing to the definitions (\ref{wr51}) and (\ref{wr80}),
we obtain from (\ref{wr36}) by It\^{o}'s formula\footnote{Recall
that we omit brackets based on the convention that $\otimes$ binds more tightly
than any other product.}
\begin{align*}
d\det\tilde F
&=(\adj \tilde F)^t . \tilde F\nabla d \phi
+ \tilde\phi^i (\adj \tilde F)^t .\nabla\partial_i d \phi \phantom{ \frac{d\tilde\lambda^2}{L^2\tilde\lambda^2} } \nonumber\\
&+\big(\tilde F^t e^1\otimes e_1 \diamond\tilde F^t e^2\otimes e_2
-\tilde F^t e^1\otimes e_2 \diamond \tilde F^t e^2\otimes e_1 \big)
\frac{d\tilde\lambda^2}{\tilde\lambda^2}\nonumber\\
&+\big( e^1\otimes e_1\otimes\tilde\phi \bullet e^2\otimes e_2\otimes\tilde\phi
- e^1\otimes e_2\otimes\tilde\phi \bullet e^2\otimes e_1\otimes\tilde\phi \big)
\frac{d\tilde\lambda^2}{L^2\tilde\lambda^2},
\end{align*}
where we already used Remark \ref{rmk:vanishing-covariation}.
By (\ref{detharm}) and (\ref{detharm02}) also the remaining quadratic variation terms
in the two last lines vanish. Since by definition of the adjunct,
$(\adj\tilde F)^t.\tilde F\nabla d\phi$ $=(\det\tilde F){\rm tr}\nabla d\phi$,
which vanishes by (\ref{ho36}), this yields (\ref{wr57}).

\medskip

\textit{Identification of $ d ( \det \tilde F )^2 $.} We claim that
\begin{equation}\label{wr57c}
d ( \det \tilde F )^2 =
(\adj\tilde F)^t \otimes\tilde\phi \bullet(\adj\tilde F)^t \otimes\tilde\phi 
\frac{d \tilde\lambda^2}{L^2 \tilde\lambda^2}
+ \text{martingale},
\end{equation}
which also implies (\ref{wr45}). Indeed, we apply It\^{o}'s formula to (\ref{wr57}), 
in which we write
\begin{align*}
\tilde\phi^i(\adj\tilde F)^t.\nabla\partial_id\phi
=(\adj\tilde F)^t \otimes\tilde\phi.\nabla^2d\phi,
\end{align*}
and insert the definition (\ref{wr80}). This yields (\ref{wr57c}).

\medskip

\textit{Identification of $ d | \tilde F |^2 $.} We claim 
\begin{align}\label{wr37}
d| \tilde F|^2= 2 (\tilde F^*\tilde F)^t .\nabla d\phi
+ 2\tilde F^{*t} \otimes\tilde\phi.\nabla^2d \phi
+\big(\frac{1}{2}| \tilde F|^2+\frac{|\tilde\phi|^2}{  2 L^2 }\big)
\frac{d\tilde\lambda^2}{\tilde\lambda^2},
\end{align}
which implies (\ref{wr43}).

\medskip

We appeal to the representation (\ref{ho38})
and apply It\^{o}'s formula in conjunction with Definition
\ref{defn:quad-var} and Remark \ref{rmk:vanishing-covariation} to (\ref{wr36}),
which yields
\begin{align*}
d|\tilde F|^2
&=2\tr \tilde F^*\tilde F\nabla d\phi+2\tilde\phi^i\tr \tilde F^*\nabla\partial_id\phi\nonumber\\
&+ \sum_{i,j=1}^2
\big(\tilde F^t e^i\otimes e_j \diamond \tilde F^t e^i\otimes e_j
\frac{d\tilde\lambda^2}{\tilde\lambda^2}
+ e^i\otimes e_j \otimes\tilde\phi \bullet e^i\otimes e_j \otimes\tilde\phi
\frac{d\tilde\lambda^2}{L^2\tilde\lambda^2}\big).
\end{align*}
Appealing to the formulas (\ref{wr38d}), in conjunction with $|\tilde F^t|^2=|\tilde F|^2$, 
and (\ref{dc03}) we obtain (\ref{wr37}).

\medskip

\textit{Identification of $ d | \tilde F |^4 $}. We claim
\begin{align}\label{wr42a}
d | \tilde F |^4
&= \frac{3}{2} | \tilde F |^4 \frac{d \tilde\lambda^2}{\tilde\lambda^2}
- 2 \det { \tilde F } \frac{d \tilde\lambda^2}{\tilde\lambda^2}
+ | \tilde F |^2 | \tilde\phi |^2 \frac{d \tilde\lambda^2}{L^2 \tilde\lambda^2} 
+ 4 \, \tilde F^{*t} \otimes \tilde\phi \bullet \tilde F^{*t} \otimes \tilde\phi
\frac{d\tilde\lambda^2}{L^2 \tilde\lambda^2} \\
&+ \text{martingale}, \nonumber
\end{align}
which implies (\ref{wr42}). Indeed, from (\ref{wr37}) and It\^{o}'s formula we obtain
$$
\begin{aligned}
d | \tilde F |^4
& = | \tilde F |^4 \frac{ d \tilde\lambda^2 }{ \tilde\lambda^2 }
 + |\tilde F|^2 |\tilde\phi|^2 \frac{ d\tilde\lambda^2 }{ L^2 \tilde\lambda^2 } \\
& + 4 (\tilde F^*\tilde F)^t \diamond (\tilde F^*\tilde F)^t
\frac{d\tilde\lambda^2}{\tilde\lambda^2}
+ 4 \, \tilde F^{*t} \otimes \tilde\phi \bullet \tilde F^{*t} \otimes \tilde\phi
\frac{d\tilde\lambda^2}{L^2 \tilde\lambda^2}
+ \text{martingale}.
\end{aligned}
$$
We now appeal to formula (\ref{wr45bis}) to the symmetric $G=(\tilde F^*\tilde F)^t$
to the effect of
\begin{equation}\label{wr42b}
(\tilde F^*\tilde F)^t \diamond (\tilde F^*\tilde F)^t
= \frac{1}{8} ( \tr (\tilde F^*\tilde F)^t )^2 
- \frac{1}{2} \det (\tilde F^*\tilde F)^t
= \frac{1}{8} | \tilde F |^4 - \frac{1}{2} ( \det \tilde F )^2,
\end{equation}
so that (\ref{wr42a}) follows.
\end{proof}


\subsection{Derivation of asymptotics}\label{section:proofs}

We now integrate the odes that we derived in the last section to prove Lemma \ref{lem:growth-proxy-phi} and the asymptotics in Theorem \ref{thm:proxy-intermittency}. We start with the former.

\begin{proof}[Proof of Lemma \ref{lem:growth-proxy-phi}.]$\mbox{}$ 
\medskip

\textit{Growth of $ \E | \tilde\phi |^2 $.} Our goal is to prove (\ref{ode02a}). 
To this end, we let $ x \coloneqq \tilde\lambda^2 $, which by (\ref{tl01}) means 
$ x = 1 + \varepsilon^2 \ln L $, and $ a \coloneqq \E| \tilde\phi |^2 $. 
Equation (\ref{wr67}) may be rewritten as
$$
x^\frac{1}{2} \frac{d}{dx} \frac{a}{x^{\frac{1}{2}}}
= \frac{da}{dx} - \frac{1}{2}\frac{a}{x}
= \frac{ \exp( \frac{2}{\varepsilon^2}( x - 1 ) ) }{ x }
$$
and integrated explicitly, i.~e.
\begin{align}\label{ho50}
a
= x^\frac{1}{2} \int_1^x dy \frac{\exp(\frac{2}{\varepsilon^2}(y-1))}{y^{\frac{3}{2}}}
= \frac{\exp(\frac{2}{\varepsilon^2}(x-1)) }{x} 
\int_1^x dy(\frac{x}{y})^{\frac{3}{2}}\exp(\frac{2}{\varepsilon^2}(y-x)).
\end{align}
To read off the asymptotics, it is convenient to change variables to 
$ - \widehat{y} = \frac{1}{\varepsilon^2} ( y - x ) $ and thus
$ - \widehat{x} = \frac{1}{\varepsilon^2} ( 1 - x ) $, 
which in view of (\ref{tl01}) is equivalent to $ \widehat{x} = \ln L \gg 1 $, so that
\begin{equation}\label{ode02}
\int_1^x dy ( \frac{ x }{ y } ) ^{ \frac{3}{2} }  \exp( \frac{2}{\varepsilon^2}( y - x ) )
= \varepsilon^2 
\int_0^{\widehat{x}} d\widehat{y}  (\frac{\varepsilon^2\widehat{x}+1}
{\varepsilon^2(\widehat{x}-\widehat{y})+1})^{\frac{3}{2}}\exp( - 2 \widehat{y} ).
\end{equation}
Due to the presence of $\exp(-2\widehat{y})$, uniformly in $\varepsilon\le 1$,
the main contribution to the integral comes from $\widehat{y}\lesssim 1$. For these $\widehat{y}$
we have $\frac{\varepsilon^2\widehat{x}+1}
{\varepsilon^2(\widehat{x}-\widehat{y})+1}\approx 1$ in our regime $\widehat{x}\gg 1$.
Hence the integral is approximated as follows
\begin{equation}\label{ho21}
\int_0^{\widehat{x}} d\widehat{y}  (\frac{\varepsilon^2\widehat{x}+1}
{\varepsilon^2(\widehat{x}-\widehat{y})+1})^{\frac{3}{2}}\exp( - 2 \widehat{y} )
\approx\int_0^\infty d\widehat{y}\exp( - 2 \widehat{y} )=\frac{1}{2},
\end{equation}
so that we learn
$$
a \approx \frac{ \varepsilon^2 }{ 2 } \frac{ \exp( \frac{2}{\varepsilon^2} (x - 1) ) }{ x }.
$$
By definition of $ x $, $ \widehat{x} $ and $ a $ this translates into (\ref{ode02a}).

\medskip

\textit{Growth of $ \E | \tilde\phi |^4 $.} We will now prove (\ref{ode03a}). 
We continue to use the abbreviations $x = \tilde\lambda^2$ 
and $a = \E | \tilde\phi |^2$ and introduce
$b \coloneqq \E| \tilde\phi |^4$ so that (\ref{wr70}) assumes
the compact form
$$
x^{ \frac{3}{2} } \frac{d}{dx} \frac{b}{x^{ \frac{3}{2}} }
= \frac{db}{dx} - \frac{3}{2} \frac{b}{x}
= 4 \frac{ \exp( \frac{2}{\varepsilon^2} ( x - 1 ) ) }{ x } a.
$$
Integrating and inserting (\ref{ode02}) yields
\begin{equation}\label{ode03}
\begin{aligned}
b
&= 4 x^{ \frac{3}{2} }
\int_1^{x} dy \frac{ \exp( \frac{2}{\varepsilon^2}( y - 1 ) ) }{ y^2 }
\int_1^y dz \frac{ \exp( \frac{2}{\varepsilon^2}( z - 1 ) ) }{ z^{ \frac{3}{2} } } \\
&= 4 \frac{ \exp( \frac{4}{\varepsilon^2}( x - 1 ) ) }{ x^2 }
\int_1^{x} dy ( \frac{ x }{ y } )^{ \frac{7}{2} } \exp( \frac{2}{\varepsilon^2}( y - x ) )
\int_1^y dz ( \frac{ y }{ z } )^{ \frac{3}{2} } \exp( \frac{2}{\varepsilon^2}( z - x ) ).
\end{aligned}
\end{equation}
We continue to transform variables according to
$ - \widehat{x} = \frac{1}{\varepsilon^2} ( 1 - x ) $, 
$ - \widehat{y} = \frac{1}{\varepsilon^2} ( y - x ) $ and now also
$ - \widehat{z} = \frac{1}{\varepsilon^2} ( z - x ) $ to the effect of
\begin{align}\label{ode03b}
&\int_1^x dy(\frac{x}{y})^{\frac{7}{2}}\exp(\frac{2}{\varepsilon^2}(y-x))
\int_1^y dz(\frac{y}{z})^{\frac{3}{2}}\exp(\frac{2}{\varepsilon^2}(z-x))\nonumber\\
&=\varepsilon^4\int_{0}^{\widehat{x}}d\widehat{y}(\frac{\varepsilon^2 \widehat{x}+1}
{\varepsilon^2(\widehat{x}-\widehat{y})+1})^{\frac{7}{2}}\exp(-2\widehat{y})
\int_{\widehat{y}}^{\widehat{x}} d\widehat{z}(\frac{\varepsilon^2(\widehat{x}-\widehat{y})+1}
{\varepsilon^2(\widehat{x}-\widehat{z})+1})^{\frac{3}{2}}\exp(-2\widehat{z}).
\end{align}
Again, in our regime $\hat x\gg 1$, the main contribution comes from
$\hat y,\hat z\lesssim 1$, so that the double integral is
\begin{align*}
\approx \int_{0}^{\infty}d\widehat{y}\exp(-2\widehat{y})
\int_{\widehat{y}}^{\infty} d\widehat{z}\exp(-2\widehat{z})
=\frac{1}{2}\int_0^\infty d\widehat{y}\exp(-4\widehat{y})=\frac{1}{8},
\end{align*} 
so that
$$
b \approx \frac{ \varepsilon^4 }{ 2 } \frac{ \exp( \frac{4}{\varepsilon^2}( x - 1 ) ) }{ x^2 }
\quad \text{as} ~ \widehat{x} \gg 1.
$$
By definition of $ x $, $ \widehat{x} $ and $ b $ this translates into (\ref{ode03a}).
\end{proof}

We note that the l.~h.~s.~integral in (\ref{ho21}) is monotone increasing in $\varepsilon$
and thus estimated as
\begin{align}\label{ho49}
& \int_0^{\widehat{x}}d\widehat{y}(\frac{\varepsilon^2\widehat{x}+1}
{\varepsilon^2(\widehat{x}-\widehat{y})+1})^\frac{3}{2}
\exp(-2\widehat{y}) \nonumber\\
&\le\int_0^{\widehat{x}}d\widehat{y}(\frac{\widehat{x}+1}{\widehat{x}-\widehat{y}+1})^\frac{3}{2}
\exp(-2\widehat{y})
\le\int_0^{\infty}d\widehat{y}(1+\widehat{y})^\frac{3}{2}
\exp(-2\widehat{y})<\infty.
\end{align}
Hence an inspection of the proof of (\ref{ode02a})
establishes the pre-asymptotic estimate (\ref{ode06a})
below. Since the inner integral in (\ref{ode03b}) is largest for $\widehat{y}=0$, 
it is also estimated by the constant on the r.~h.~s.~of (\ref{ho49});
the outer integral is then treated like (\ref{ho49}) with the exponent $\frac{3}{2}$
replaced by $\frac{7}{2}$. This yields the pre-asymptotic estimate (\ref{ode06}).

\begin{remark}\label{rmk:growth-proxy-phi-ub}[{see \cite[Lemma 4]{CMOW}}].
We have for $\varepsilon\le 1$
\begin{align}
\E | \tilde\phi |^2 & \lesssim \varepsilon^2 \frac{ L^2 }{ \tilde\lambda^2 }, \label{ode06a}\\
\E | \tilde \phi |^4 &\lesssim \varepsilon^4 \frac{ L^4 }{ \tilde\lambda^4 }. \label{ode06}
\end{align}
%
\end{remark}

We collected all ingredients to prove the asymptotics stated in Theorem \ref{thm:proxy-intermittency}.
 
\begin{proof}[Proof of asymptotics in Theorem \ref{thm:proxy-intermittency}.] $\mbox{}$
\medskip

\textit{Asymptotics of $ \E | \tilde F |^2 $.} 
As in the proof of Lemma \ref{lem:growth-proxy-phi}, we use the convention $ x = \tilde\lambda^2 $ and now consider
$ A \coloneqq \E | \tilde F |^2 $ to rewrite (\ref{wr43}) as
$$
x^\frac{1}{2} \frac{d}{dx} \frac{A}{x^\frac{1}{2}}
= \frac{dA}{dx} - \frac{1}{2} \frac{A}{x}
=  \frac{1}{2} \frac{a}{x} \exp( - \frac{2}{\varepsilon^2} ( x - 1 ) ).
$$
By virtue of (\ref{ho50}) and the initial condition $ F = \rm id $ and thus $A=2$ at $ x = 1 $ we have
$$
\begin{aligned}
\frac{A}{x^\frac{1}{2}} - 2
&= \frac{1}{2} \int_1^x dy \frac{a}{ y^{ \frac{3}{2} } } \exp( - \frac{2}{\varepsilon^2} ( y - 1 ) ) \\
&\stackrel{(\ref{ho50})}{=} \frac{1}{2} \int_1^x dy \frac{ 1 }{ y^{ \frac{5}{2} } }
\int_1^y dz ( \frac{ y }{ z } )^{\frac{3}{2}} \exp( \frac{2}{\varepsilon^2} ( z - y ) ).
\end{aligned}
$$
We continue to change the inner variable according to
$ - \widehat{z} = \frac{1}{\varepsilon^2} ( z - y ) $,
but change the outer one according to
$ - \widehat{y} = \frac{1}{\varepsilon^2} (1 - y) $,
rather like
$ - \widehat{x} = \frac{1}{\varepsilon^2} (1 - x) $.
This yields
\begin{equation}\label{ode10}
\frac{A}{x^\frac{1}{2}}-2
= \frac{\varepsilon^2}{2} \int_{ 0 }^{ \widehat{x} } d \widehat{y} \frac{ \varepsilon^2 }{ ( \varepsilon^2 \widehat{y} + 1 )^{\frac{5}{2}} }
\int_{0}^{\widehat{y}} d\widehat{z} (\frac{\varepsilon^2\widehat{y}+1}
{\varepsilon^2(\widehat{y}-\widehat{z})+1})^{\frac{3}{2}}\exp(-2\widehat{z}).
\end{equation}
By (\ref{ho49}), the inner integral is $\lesssim 1$.
Hence the double integral is
\begin{align*}
\lesssim\int_{0}^{\infty}d\widehat{y}
\frac{\varepsilon^2}{(\varepsilon^2\widehat{y}+1)^\frac{5}{2}}
=\int_1^\infty dy\frac{1}{y^\frac{5}{2}} <\infty.
\end{align*}
Therefore in the regime $\varepsilon\ll 1$, (\ref{ode10}) yields
$$
A \approx 2 {x^\frac{1}{2}},
$$
which translates back to (\ref{mr05}).

\medskip

\textit{Rough upper bound on $ \E | \tilde F |^4 $.} 
By the universality of $\bullet$ established in Appendix \ref{ss:var-nabla-nabla-phi}, 
the last term in (\ref{wr42})
is estimated by the penultimate one:
\begin{align*}
\E(\tilde F^{*t}\otimes\tilde\phi\bullet\tilde F^{*t}\otimes\tilde\phi)
\lesssim\mathbb{E}|\tilde\phi|^2|\tilde F|^2;
\end{align*}
by Cauchy-Schwarz and (\ref{ode06}), the latter is estimated as follows
\begin{align*}
\frac{1}{L^2}\mathbb{E}|\tilde\phi|^2|\tilde F|^2\lesssim\frac{\varepsilon^2}{\tilde\lambda^2}
\mathbb{E}^\frac{1}{2}|\tilde F|^4.
\end{align*}
Hence with the abbreviations $B\coloneqq\E|\tilde F|^4$ and $C\coloneqq\E(\det\tilde F)^2$,
(\ref{wr42}) yields
\begin{align}\label{ho52}
|\frac{dB}{dx}-(\frac{3}{2}B-2C)\frac{1}{x}|\lesssim\varepsilon^2\frac{B^\frac{1}{2}}{x^2}.
\end{align}
Since $C\ge 0$, (\ref{ho52}) in particular entails the differential inequality
$$
x^{\frac{3}{2}} \frac{d}{dx} \frac{B}{x^{ \frac{3}{2} }}
=\frac{dB}{dx}-\frac{3}{2}\frac{B}{x}
\lesssim \varepsilon^2 \frac{ B^{\frac{1}{2}} }{ x^2 }
=\varepsilon^2\frac{1}{x^\frac{5}{4}}(\frac{B}{x^\frac{3}{2}})^\frac{1}{2},
$$
which we rewrite as
$$
\frac{d}{dx}(\frac{B}{x^{\frac{3}{2}}})^\frac{1}{2}
\lesssim\varepsilon^2\frac{1}{x^\frac{11}{4}}.
$$
Integration with the initial condition $B=4$ at $x=1$ gives
\begin{align}\label{ho54}
(\frac{B}{x^{\frac{3}{2}}})^\frac{1}{2}-2\lesssim\varepsilon^2
\quad\mbox{and thus}\quad
B\lesssim x^\frac{3}{2}.
\end{align}

\medskip

\textit{Asymptotics of $ \E (\det\tilde F-1)^2 $.} 
We recall the definitions $B=\E|\tilde F|^4$ and $C=\E(\det\tilde F)^2$;
by the same argument as above, that is, universality of $\bullet$, Cauchy-Schwarz, 
and (\ref{ode06}), the ODE (\ref{wr45}) yields
\begin{align*}
|\frac{dC}{dx}|\lesssim\varepsilon^2\frac{B^\frac{1}{2}}{x^2}.
\end{align*}
By the upper bound (\ref{ho54}), this implies
\begin{align}\label{ho56}
|\frac{dC}{dx}|\lesssim\varepsilon^2\frac{1}{x^\frac{5}{4}}.
\end{align}
Integration with the initial condition $C=1$ at $x=1$ gives
\begin{align}\label{ho55}
|C-1|\lesssim\varepsilon^2,
\end{align}
which turns into (\ref{mr04}) in view of $\E\det\tilde F=1$ as a consequence of
(\ref{wr44}).

\medskip

\textit{Asymptotics of $ \E | \tilde F |^4 $.} 
We return to (\ref{ho52}) and combine it with (\ref{ho56}) to obtain
a differential inequality rewrite on $\frac{3}{2}B-2C$, 
and use (\ref{ho54}) to estimate the r.~h.~s.:
\begin{align*}
|\frac{d}{dx}(\frac{3}{2}B-2C)-\frac{3}{2}(\frac{3}{2}B-2C)\frac{1}{x}|
\lesssim\varepsilon^2\frac{1}{x^\frac{5}{4}}.
\end{align*}
We rewrite this once more as
\begin{align*}
|\frac{d}{dx}\frac{1}{x^\frac{3}{2}}(\frac{3}{2}B-2C)|
\lesssim\varepsilon^2\frac{1}{x^\frac{11}{4}}.
\end{align*}
Integration with the initial condition $B=4$ and $C=1$ at $x=1$ gives
\begin{align*}
|\frac{1}{x^\frac{3}{2}}(\frac{3}{2}B-2C)-4|\lesssim\varepsilon^2,
\end{align*}
which together with (\ref{ho55}) yields
\begin{align*}
|\frac{1}{x^\frac{3}{2}}(\frac{3}{2}B-2)-4|\lesssim\varepsilon^2.
\end{align*}
It is the weaker version
\begin{align*}
B\approx \frac{8}{3}x^\frac{3}{2}+\frac{4}{3}
\end{align*}
that corresponds to (\ref{mr06}).
\end{proof}


\subsection{Proof of main result}\label{section:proof-mr} 

We now complete the proof of Theorem \ref{thm:proxy-intermittency} by establishing \eqref{ao13}. The asymptotics were already established in Section \ref{section:proofs}.

\begin{proof}[Proof of Proposition \ref{prop:proxy-intermittency-2}
--- Non equi-integrability of $ ( \E | \tilde F |^2 )^{ - \frac{1}{2} } \tilde F $.]
The same It\^{o} argument that lead from (\ref{wr37}) to (\ref{wr42}) 
also implies that
that for any test function/observable $\zeta$ of $r=|\tilde F|^2$
\begin{align}\label{ho19}
d\mathbb{E}\zeta(|\tilde F|^2)
&=\mathbb{E}\Big( \frac{d\zeta}{dr}
(|\tilde F|^2)\big(\frac{1}{2}|\tilde F|^2+\frac{|\tilde\phi|^2}{
2L^2})
\nonumber\\
&+\frac{1}{2} \frac{d^2\zeta}{dr^2} (|\tilde F|^2) 
\big( \frac{1}{2}|\tilde F|^4-2(\det\tilde F)^2 
+\frac{ ( 2\tilde F^{*t } \otimes\tilde\phi ) \bullet ( 2\tilde F^{*t } \otimes\tilde\phi ) }{L^2}\big)\Big)
\frac{d\tilde\lambda^2}{\tilde\lambda^2};
\end{align}
in fact, (\ref{wr42}) amounts to (\ref{ho19}) for $\zeta(r)=r^2$.
Since according to Appendix \ref{ss:var-nabla-nabla-phi}, 
the form $\bullet$ is universal, we have
\begin{align*}
| \tilde F^{*t } \otimes\tilde\phi \bullet \tilde F^{*t } \otimes\tilde\phi |
\lesssim|\tilde\phi|^2|\tilde F|^2\quad\mbox{next to}\quad
|{\rm det}\tilde F|\lesssim|\tilde F|^2.
\end{align*}
Hence we obtain the differential inequality (where $r=|\tilde F|^2$ is the argument of
$\zeta$ and its derivatives)
\begin{align}\label{ao01}
\Big|\frac{d}{d\tau}\mathbb{E}\zeta
-\mathbb{E}\big(\frac{1}{2}|\tilde F|^2\frac{d\zeta}{dr}
+\frac{1}{4}|\tilde F|^4\frac{d^2\zeta}{dr^2}\big)\Big|
\lesssim\mathbb{E}\Big( |\det \tilde F| |\tilde F|^2|\frac{d^2\zeta}{dr^2}|
+\frac{|\tilde\phi|^2}{L^2}\big(|\frac{d\zeta}{dr}|+|\tilde F|^2|\frac{d^2\zeta}{dr^2}|\big)\Big)
\end{align}
in terms of the time and space-like variable
\begin{align}\label{ao04}
\tau=\ln\tilde\lambda^2 \quad\mbox{and}\quad r=|\tilde F|^2 ,
\end{align}
respectively.

\medskip

This motivates to consider a $\tau$-dependent $\zeta=\zeta(\tau,r)$
that satisfies the backward equation
\begin{align}\label{ao02}
\frac{\partial\zeta}{\partial\tau}+\frac{r}{2}\frac{\partial\zeta}{\partial r}
+\frac{r^2}{4}\frac{\partial^2\zeta}{\partial r^2}=0
\end{align}
so that by the (pre-asymptotic w.~r.~t.~$L$) bounds
(\ref{mr04}) \& (\ref{mr05}) and (\ref{ode06a}), and using the notation
(\ref{ao04}), (\ref{ao01}) implies
\begin{align*}
|\frac{d}{d\tau}\mathbb{E}\zeta( \tau, |\tilde F|^2)|
\lesssim\sup_{r}r|\frac{\partial^2\zeta}{\partial r^2}|
+\varepsilon^2e^{-\tau}\sup_{r}|\frac{\partial\zeta}{\partial r}|.
\end{align*}
We will use this after integration in $\tau'\in(0,\tau)$ and appealing to $|\tilde F_{ \tau = 0 }|^2=2$, see (\ref{wr04}):
\begin{align}\label{ao03}
\mathbb{E}\zeta(\tau,|\tilde F|^2)\lesssim\zeta(0,2)
+\int_0^\tau d\tau'\sup_{r}r|\frac{\partial^2\zeta}{\partial r^2}(\tau',r)|
+\varepsilon^2\int_0^\tau d\tau'e^{-\tau'}
\sup_{r}|\frac{\partial\zeta}{\partial r}(\tau',r)|.
\end{align}

\medskip

We now construct a suitable solution $\zeta$ of the backwards parabolic equation (\ref{ao02}),
which amounts to a duality argument in the jargon of PDE theory. To this purpose, 
we pass to the natural\footnote{Note that in view of (\ref{ao04}), this amounts to the natural scale $r=\tilde\lambda\hat r$
for $r=|\tilde F|^2$; $\zeta=\hat r\hat\zeta$ means that $\hat\zeta$ modulates
the properly rescaled second moment of $\tilde F$.} variables
\begin{align}\label{ao06}
r=e^\frac{\tau}{2}\hat r \stackrel{(\ref{ao04})}{=}\tilde\lambda\hat r
\quad\mbox{and}\quad\zeta=\hat r\hat\zeta,
\end{align}
in the sense of $\zeta(\tau,e^\frac{\tau}{2}\hat r)$ $=\hat r\hat\zeta(\tau,\hat r)$,
which implies $\frac{\partial\zeta}{\partial\tau}+\frac{r}{2}\frac{\partial\zeta}{\partial r}$ 
$=\hat r\frac{\partial\hat\zeta}{\partial\tau}$.
Hence by $r^2\frac{\partial}{\partial r^2}$ $=\hat r^2\frac{\partial^2}{\partial \hat r^2}$
and 
$\frac{\partial^2\zeta}{\partial\hat r^2}$ $=\hat r\frac{\partial^2\hat\zeta}{\partial\hat r^2}
+2\frac{\partial\hat\zeta}{\partial\hat r}$, the form of (\ref{ao02}) is preserved:
\begin{align}\label{ao05}
\frac{\partial\hat\zeta}{\partial\tau}+\frac{\hat r}{2}\frac{\partial\hat\zeta}{\partial\hat r}
+\frac{\hat r^2}{4}\frac{\partial^2\hat\zeta}{\partial\hat r^2}=0.
\end{align}

\medskip

The scale-invariant (in $\hat r$) nature of (\ref{ao05}) suggests to pass to 
logarithmic independent variable
\begin{align}\label{ao10}
\hat\sigma=\ln\hat r.
\end{align}
Indeed, because of $\hat r\frac{\partial}{\partial\hat r}=\frac{\partial}{\partial\hat\sigma}$ and
$\hat r^2\frac{\partial}{\partial \hat r^2}$
$=\frac{\partial^2}{\partial\hat\sigma^2}-\frac{\partial}{\partial\hat\sigma}$,
(\ref{ao05}) turns into the constant-coefficient equation
\begin{align}\label{ao07}
\frac{\partial\hat\zeta}{\partial\tau}+\frac{1}{4}\frac{\partial\hat\zeta}{\partial\hat\sigma}
+\frac{1}{4}\frac{\partial^2\hat\zeta}{\partial\hat\sigma^2}=0.
\end{align}
For momentarily fixed $\tau\gg 1$ and $\hat\sigma\gg 1$ (as above we use $ \tau' \in ( 0, \tau ) $ and $ \hat\sigma' $ as active variables instead) we consider the $C^2$ 
terminal\footnote{as opposed to initial, in view of the backwards nature of (\ref{ao07})} data 
given by
\begin{align}
& \hat\zeta(\tau,\hat\sigma')\left\{\begin{array}{ll}
=1&\mbox{for}\;\hat\sigma'\le\hat\sigma\\
\in[0,1]&\mbox{for all}\;\hat\sigma'\\
=0&\mbox{for}\;\hat\sigma'\ge\hat\sigma+1\end{array}\right\} \quad\mbox{and}\quad
|\frac{\partial^2\hat\zeta}{d\hat\sigma'^2}(\tau,\hat\sigma')| \lesssim 1.
\label{ao08}
\end{align}
These terminal data are chosen such that the l.~h.~s.~of (\ref{ao03}) controls
the second moments of $\tilde F$ up to tails; indeed, by the definitions
(\ref{ao02}), (\ref{ao06}), and (\ref{ao10}), we have
\begin{align}\label{ao14}
\mathbb{E}\zeta(\tau,|\tilde F|^2)
\ge\frac{1}{\tilde\lambda}\mathbb{E}I(|\tilde F|^2 \le\tilde\lambda\hat r)|\tilde F|^2.
\end{align}

\medskip

We now choose the parameter $\hat\sigma=\ln\hat r$ 
so large that the first r.~h.~s.~term of (\ref{ao03}) satisfies
\begin{align}\label{ao09}
\zeta( \tau' = 0, r' = 2 )
\stackrel{(\ref{ao06}),(\ref{ao10})}{=}
2 \hat\zeta( \tau' = 0, \hat\sigma' = \ln 2)\ll 1.
\end{align}
To this purpose we note that the (translation invariant) semi-group kernel for 
(\ref{ao07}) with terminal data at $\tau'=\tau$ is a 
Gaussian of mean $-\frac{1}{4}(\tau-\tau')$ and variance $\frac{1}{2}(\tau-\tau')$.\footnote{This is conveniently seen by making a Gaussian Ansatz with general mean and variance that lead to the specific form claimed above.} 
Hence in view of (\ref{ao08}) in form of
$\hat\zeta(\tau,\hat\sigma')$ $\le I(\hat\sigma'<\hat\sigma+1)$, 
we have in terms of the error function
\begin{align*}
\hat\zeta(\tau',\hat\sigma')\le F_{ \mathcal{N}(0,1) } 
\Big(\frac{-\hat\sigma'+\hat\sigma+1-\frac{1}{4}(\tau-\tau')}
{\sqrt{ \frac{1}{2}(\tau-\tau')}}\Big)
\end{align*}
and thus $\hat\zeta( \tau' = 0, \hat\sigma' = \ln 2)$ 
$\le  F_{ \mathcal{N}(0,1) } ( \frac{1- \ln 2 +\hat\sigma - \frac{\tau}{4}}
{ \sqrt{\frac{\tau}{2}} } ) $. Hence (\ref{ao09}) is satisfied provided
$\frac{\tau}{4}-\hat\sigma$ $\gg\sqrt{\tau}$, which in view (\ref{ao04}) and (\ref{ao10})
amounts to (\ref{ao12}).

\medskip

We now turn to the two last r.~h.~s.~terms of (\ref{ao03}) and shall establish
\begin{align}
\int_0^\tau d\tau'
\sup_{r}r|\frac{\partial^2\zeta}{\partial r^2}(\tau',r)|
&\lesssim\frac{1}{\sqrt{\tau}},\label{ao15}\\
\int_0^\tau d\tau'e^{-\tau'}
\sup_{r}|\frac{\partial\zeta}{\partial r}(\tau',r)|
&\lesssim 1.\label{ao16}
\end{align}
For this we note that by (\ref{ao06}) and (\ref{ao10})
\begin{align*}
\frac{\partial\zeta}{\partial r}
&=e^{-\frac{\tau'}{2}}\frac{\partial}{\partial\hat r}\hat r\hat\zeta
=e^{-\frac{\tau'}{2}}(\frac{\partial}{\partial\hat\sigma'}+1)\hat\zeta,\\
r\frac{\partial^2\zeta}{\partial r^2}
&=e^{-\frac{\tau'}{2}}\hat r\frac{\partial^2}{\partial\hat r^2}\hat r\hat\zeta
=e^{-\frac{\tau'}{2}}(\frac{\partial}{\partial\hat\sigma'}+1)
\frac{\partial\hat\zeta'}{\partial\hat\sigma'}.
\end{align*}
In order to control suprema in $\sigma'$ of $(\frac{\partial}{\partial\hat\sigma}+1)\hat\zeta$ 
and $(\frac{\partial}{\partial\hat\sigma}+1)
\frac{\partial\hat\zeta}{\partial\tilde\sigma}$ as a function of $\tau'$, 
the drift term in (\ref{ao07}) is irrelevant, so that we may appeal to standard
estimates on the (backward) heat equation. From these we learn that the estimates
on the terminal data, cf.~(\ref{ao08}), both uniform 
$\sup_{\hat\sigma'}(|\hat\zeta|+|\frac{\partial\hat\zeta}{\partial\hat\sigma'}|
+|\frac{\partial^2\hat\zeta}{\partial\hat\sigma'^2}|)$ $\lesssim 1$ 
and integral $\int d\hat\sigma'(|\frac{\partial\hat\zeta}{\partial\hat\sigma'}|
+|\frac{\partial^2\hat\zeta}{\partial\hat\sigma'^2}|)$ $\lesssim 1$, lead to
\begin{align*}
|\hat\zeta|\lesssim 1,\quad
|\frac{\partial\hat\zeta}{\partial\hat\sigma'}|\lesssim\frac{1}{\sqrt{1+\tau-\tau'}},\quad
|\frac{\partial^2\hat\zeta}{\partial\hat\sigma'^2}|\lesssim\frac{1}{1+\tau-\tau'}.
\end{align*}
These estimates immediately imply (\ref{ao16}). For (\ref{ao15}), we appeal to
$\int_0^\tau d\tau'e^{-\frac{\tau'}{2}}$ $\frac{1}{\sqrt{1+\tau-\tau'}}$ 
$\lesssim\frac{1}{\sqrt{\tau}}$, which is elementary.

\medskip

In view of (\ref{mr05}) and $\varepsilon^2\ll 1$, (\ref{ao13}) follows from inserting (\ref{ao14}), 
(\ref{ao09}), (\ref{ao15}), and (\ref{ao16}) into (\ref{ao03}).
\end{proof}


Finally, we collected all ingredients to establish Theorem \ref{thm:intermittency}.

\begin{proof}[Proof of Theorem \ref{thm:intermittency}.] $\mbox{}$ 
\medskip

\textit{Asymptotics of $ \E | \det F | $.} 
Applying Cauchy-Schwarz to the elementary inequality
$|\det F-\det\tilde F|$
$\lesssim(|F|+|\tilde F|)|F-\tilde F|$
$\lesssim|\tilde F||F-\tilde F|+|F-\tilde F|^2$ we obtain
\begin{align}\label{eqdet01}
\E | \det F - \det \tilde F |
\lesssim ( \E| \tilde F |^2 )^{ \frac{1}{2} }( \E | F - \tilde F |^2 )^{ \frac{1}{2} }
+\E | F - \tilde F |^2.
\end{align}
In order to establish the estimate (\ref{wr03bis}) on $\E | F - \tilde F |^2$, 
we apply the divergence to the decomposition (\ref{wr07}) to obtain
$$
\nabla \cdot a( { \rm id } + \nabla \tilde\phi )
= \nabla \cdot f
$$
so that by \eqref{cw14} and the stationarity of $\tilde\phi$
$$
\nabla \cdot a ( \nabla\phi - \nabla\tilde\phi ) = - \nabla \cdot f
\quad\mbox{and}\quad\mathbb{E}\nabla\phi=\mathbb{E}\nabla\tilde\phi=0.
$$ 
Hence because of the form (\ref{ho15}) of $a$ (which implies $\xi\cdot a\xi=|\xi|^2$), 
the energy estimate on the level of stationary gradient fields of
vanishing expectation takes the form
\begin{align*}
\E | \nabla\phi - \nabla\tilde\phi |^2 \leq \E | f |^2,
\end{align*}
which by the definitions (\ref{cw17b}) \& (\ref{cw17bis}) 
and the estimate (\ref{wr03}) yields (\ref{wr03bis}).
%
%

\medskip 

Inserting (\ref{wr03bis}) and (\ref{mr05}) 
into (\ref{eqdet01}), we obtain in the regime of $\varepsilon\ll 1$
\begin{align}\label{ho16}
\E | \det F - \det \tilde F |
\lesssim(\varepsilon+\varepsilon^2)\tilde\lambda\sim\varepsilon\tilde\lambda.
\end{align}
Note that this type of argument also implies \eqref{cw18}, see \cite[(103)]{CMOW}.
Using $\E|\det F|$ $\le\E | \det F - \det \tilde F |$ 
$+(\E(\det \tilde F - 1)^2)^\frac{1}{2}$ $+1$,
into which we insert (\ref{ho16}) and (\ref{mr04}), 
we obtain in the regime of $\varepsilon\ll 1$
\begin{align*}
\E|\det F|\lesssim \varepsilon\tilde\lambda+1,
\end{align*}
which in view of (\ref{cw18}), 
also in form of $\tilde\lambda\approx\lambda$, assumes the desired form of (\ref{mr02}).

\medskip

\textit{The failure of equi-integrability for $ ( \E | F |^2 )^{ - \frac{1}{2} } F $.}
We start from the elementary inequality
$$
\frac{1}{2} |F|^2 I ( |F|^2 \leq \frac{1}{4} r )
\leq |\tilde F|^2 I ( |\tilde F|^2 \leq r ) + |F - \tilde F|^2.
$$
Choosing $ r = \hat r \E|\tilde F|^2 $ with 
$\hat r$ as in (\ref{ao12})
and taking the expectation, we obtain from (\ref{wr03bis}) and (\ref{ao13}) in the regime $\varepsilon\ll 1$
$$
\E|F|^2 I ( |F|^2 \leq \frac{1}{4} \hat r \E|\tilde F|^2 )
\ll\mathbb{E}|\tilde F|^2+\tilde\lambda\stackrel{(\ref{mr05})}{\sim}\tilde\lambda.
$$
In view of (\ref{cw18}), this implies the desired (\ref{ao17}).
\end{proof}

\begin{proof}[Proof of Theorem \ref{th:forFabio}] Let us abbreviate
\begin{align*}
D_T = \int_T^{2T} dt | \nabla \bar X_t |^2
\end{align*}
throughout the proof.

\medskip

\textit{Failure of equi-integrability for $ ( \E D_T )^{-1} D_T $.} We will establish a more precise version of \eqref{so10}, namely
\begin{align}\label{so10precise}
\mathbb{E}D_TI(D_T\le(\varepsilon^2 \ln T)^\alpha \mathbb{E}D_T)\le\frac{1}{2}\mathbb{E}D_T,
\end{align}
for $\varepsilon^2\ln T\gg_\alpha 1$ and $\varepsilon\ll_\alpha 1$.

\medskip

As in the proof of Theorem \ref{thm:intermittency}, we start from the elementary inequality
$$
D_T I ( | D_T | \leq \frac{1}{2} r ) \leq \tilde D_T  I ( | \tilde D_T | \leq r ) + | D_T - \tilde D_T |.
$$
Choosing $ \tilde D_t = T | \tilde F_L |^2 $ and $ r = T \E | \tilde F_L |^2 \widehat{r} $ we obtain
\begin{align}\label{ff01}
\begin{aligned}
\E D_T I ( | D_T | \leq \frac{1}{2} T \E | \tilde F_L |^2 \widehat{r} ) & \leq T \E | \tilde F_L |^2  I ( | \tilde F_L |^2 \leq \E | \tilde F_L |^2 \widehat{r}  ) \\
& + \E | \int_T^{2T} dt | \nabla \bar X_t |^2 - \int_T^{2T} dt | \tilde F_L  |^2 |.
\end{aligned}
\end{align}
The last term in \eqref{ff01} may be estimated by
\begin{align*}
\E | \int_{T}^{2T} dt | \nabla \bar X_t |^2 - \int_T^{2T} dt | \tilde F_L  |^2 |
&\leq \E \int_T^{2T} dt | ( \nabla \bar X_t - \tilde F_L  ) \cdot ( \nabla \bar X_t + \tilde F_L  ) | \\
&\leq \left( \E \int_T^{2T} dt | \nabla \bar X_t  + \tilde F_L  |^2 \right)^{ \frac{1}{2} } \left( \E \int_T^{2T} dt | \nabla \bar X_t - \tilde F_L  |^2 \right)^{ \frac{1}{2} } .
\end{align*}
By fixing $ L^2 = T $, we can invoke \eqref{so11}, \eqref{so13} and \eqref{mr05}, to further estimate
\begin{align}\label{ff02}
\E | \int_T^{2T} dt | \nabla \bar X_t |^2 - \int_T^{2T} dt | \tilde F_L  |^2 |
&\lesssim \varepsilon \tilde\lambda_L T,
\end{align}
so that \eqref{ff01} implies
\begin{align*}
\E D_T I ( | D_T | \leq \frac{1}{2}T \E | \tilde F |^2 \widehat{r} )
\lesssim T \E | \tilde F_L |^2  I ( | \tilde F_L |^2 \leq \E | \tilde F_L |^2 \widehat{r}  ) + \varepsilon \tilde\lambda T.
\end{align*}
By virtue of \eqref{so11} and \eqref{mr05}, the last inequality and \eqref{ao13} imply \eqref{so10precise}.

\medskip

\textit{Asymptotics of $ \int_{T}^{2T} dt | \det \nabla \bar X_t | $.} Integrating \eqref{eqdet01} for $ F = \nabla \bar X $ in time, we obtain by Cauchy-Schwarz in time
\begin{align*}
\E \int_T^{2T} dt | \det \nabla \bar X_t - \det \tilde F_L |
&\lesssim \left( \E \int_T^{2T} dt | \tilde F_L |^2 \right)^{ \frac{1}{2} } \left( \E \int_{T}^{2T} dt | \nabla \bar X_t - \tilde F_L |^2 \right)^{ \frac{1}{2} } \\
&+ \E \int_{T}^{2T} dt | \nabla \bar X_t - \tilde F_L |^2,
\end{align*}
where as above $ L^2 = T $. Hence, by \eqref{so13} and \eqref{mr05} we obtain
\begin{align*}
\E \int_T^{2T} dt | \det \nabla \bar X_t - \det \tilde F_L | \lesssim ( \varepsilon + \varepsilon^2 ) \tilde\lambda_L  T \lesssim \varepsilon \tilde\lambda_L T .
\end{align*}
By triangle inequality and Cauchy-Schwarz we obtain
\begin{align*}
\E \int_T^{2T} dt | \det \nabla \bar X_t | 
\leq \E \int_T^{2T} dt | \det \nabla \bar X_t - \det \tilde F_L |
+ \sqrt{T} ( \E \int_T^{2T} dt ( \det \tilde F_L - 1 )^2 )^{ \frac{1}{2} }
+ T.
\end{align*}
Combing the previous two lines and \eqref{mr04} we estimate
\begin{align*}
\E \int_T^{2T} dt | \det \nabla \bar X_t | 
\leq \varepsilon \tilde\lambda_L  T + \varepsilon T + T.
\end{align*}
In the regime $ \varepsilon \ll 1$ and $ \varepsilon^2 \ln T \gg 1 $, we may use \eqref{so11} to conclude \eqref{ff03}.
\end{proof}

\appendix


\section{Identification of quadratic variation of the drivers}
\label{ref:quadratic-variation}\label{appendix:quad-var}

In this appendix, we identify the quadratic variations of $\nabla d \phi $ and $ \nabla^2 d \phi $.
More specifically, we diagonalize the positive semi-definite quadratic forms
$\diamond$ and $\bullet$ introduced in Definition \ref{defn:quad-var}.
In particular, we will learn that both are universal\footnote{i.~e.~independent of any parameter
like $L$ or $\varepsilon$}. This is an exercise in representation theory for the group of rotations,
which we reproduce for the convenience of the reader.
It is natural to start with $ \nabla d \phi $ in view of its
definition via $ d \psi $, see (\ref{wr01}), which is fully characterized by (\ref{in09}).


\subsection{Quadratic variation of $\nabla d\phi$}\label{ss:qv1}\label{ss:var-nabla-phi}

We will now identify $\diamond$ arising from (\ref{wr51}).
A less representation-theoretic argument for the contraction (\ref{wr38d})
was given in 
\cite[(26)]{SuppCMOW}. 
In view of the canonical isomorphism 
\begin{align*}
(\mbox{cotangent space})\otimes(\mbox{tangent space})
\cong\mbox{endomorphisms of tangent space},
\end{align*}
we (also) view $\diamond$ as a symmetric bilinear form on the latter.
By isotropy, $\diamond$ is invariant under the action of the rotation group
via conjugation of the endomorphisms: 
%
\begin{align}\label{ho28}
Q.G=QGQ^{-1}=QGQ^*\quad\mbox{or equivalently}\quad Q.(\xi\otimes\dot x)=Q^{-t}\xi\otimes Q\dot x.
\end{align}
This group action leaves the following decomposition (\ref{wr50})
of endomorphism space invariant.
%
%
This motivates to introduce the following basis on this four-dimensional space 
\begin{align}\label{wr82}
\begin{array}{ll}
E^1 :=e^1\otimes e_1-e^2\otimes e_2,\;
&E^2:=e^1\otimes e_2+e^2\otimes e_1,\\[1ex]
E^3:=e^1\otimes e_2-e^2\otimes e_1,\;
&E^4:=e^1\otimes e_1+e^2\otimes e_2.
\end{array}
\end{align}
Indeed, $\{E^1,E^2\}$ is a basis of the space of trace-free symmetric endomorphisms,
$E^3=J$ is a skew endomorphism\footnote{We denote the rotation by $\frac{\pi}{2}$
(in the positive direction) on tangent and cotangent space both with $J$.}, 
and $E^4={\rm id}$ is an isotropic endomorphism.
This choice is natural in many ways: 
$E^1,\hdots,E^4$ are orthogonal
w.~r.~t.~the inner product $:$ defined in (\ref{ho40}), more precisely
\begin{align}\label{ho41}
E^m:E^n=0\quad\mbox{unless}~ m=n\quad\mbox{and}\quad
E^m:E^m=2,
\end{align}
as used in the argument for (\ref{wr38d}).

\medskip

We claim that this basis is orthogonal also w.~r.~t.~$\diamond$, i.~e.
\begin{align}
E^m\diamond E^n=0\quad\mbox{unless}\quad m=n\quad\mbox{and}\quad\label{wr52}\\
2 E^1 \diamond E^1
= 2 E^2 \diamond E^2
= E^3 \diamond E^3
= 1
\quad\mbox{and}\quad
E^4 \diamond E^4=0.\label{wr53}
\end{align}

\medskip

We start with the first component in (\ref{wr50}) by establishing conformality:
\begin{align}\label{wr48}
E^1\diamond E^1=E^2\diamond E^2\quad\mbox{and}\quad E^1\diamond E^2=0.
\end{align}
We give an abstract argument for (\ref{wr48}) which
seems a bit pretentious here but is useful
when appealing to it in Subsection \ref{ss:var-nabla-nabla-phi}. 
Let us denote by $Q_\theta$ the rotation by $ \frac{1}{2} \theta$; 
the factor of $\frac{1}{2}$ in front of the angle is natural in view of 
$(\mbox{trace free symmetric})$ $=(\mbox{constant})\times(\mbox{reflection})$,
since angles between reflection lines should be measured mod $\pi$. In view of
\begin{align}\label{wr48bis}
\frac{dQ_{\theta}}{d\theta}\Big|_{\theta=0}= \frac{1}{2} J\quad\mbox{and}\quad
\frac{dQ_{\theta}^{-1}}{d\theta}\Big|_{\theta=0}=\frac{dQ_{-\theta}}{d\theta}\Big|_{\theta=0}
=- \frac{1}{2} J,
\end{align}
the generator of the action (\ref{ho28}) is given by
\begin{align*}
\frac{d}{d\theta}\Big|_{\theta=0}Q_\theta.G= \frac{1}{2} (JG-GJ).
\end{align*}
The definitions (\ref{wr82}) of $E^1,E^2$ are made such that they are related by the generator:
\begin{align}\label{ho29}
\frac{d}{d\theta}\Big|_{\theta=0}Q_\theta.E^1=E^2
\end{align}
By the action properties $Q_{\theta}^{-1}.G=Q_{-\theta}.G$ and 
$Q_{\theta+\theta'}.G$ $=Q_\theta.Q_{\theta'}.G$,
(\ref{ho29}) entails
\begin{align}\label{ho30}
\frac{d}{d\theta}\Big|_{\theta=0}Q_\theta.E^2=-E^1\quad\mbox{and}\quad
\frac{d^2}{d\theta^2}\Big|_{\theta=0}Q_\theta.E^1= - E^1.
\end{align}
We obtain the second part of (\ref{wr48}) from applying $\frac{d}{d\theta}\big|_{\theta=0}$
to $Q_\theta E^1\diamond Q_\theta E^1=E^1\diamond E^1$, which follows from the invariance
of $\diamond$, and then appealing to (\ref{ho29}) and the symmetry of $\diamond$. 
Likewise, we obtain the first part of (\ref{wr48}) 
from applying $\frac{d^2}{d\theta^2}\big|_{\theta=0}$ 
and appealing to (\ref{ho29}) and (\ref{ho30}).

\medskip

We now consider the mixed terms w.~r.~t.~(\ref{wr50}). Since $E^3$ is a fixed point 
under conjugation (with rotations),
$G\mapsto E^3\diamond G$ is an invariant linear form on the (dual of) endomorphism space,
and thus is a multiple of the trace. Hence the first and the second components of (\ref{wr50})
are orthogonal under $\diamond$:
\begin{align}\label{wr30}
E^1\diamond E^3=E^2\diamond E^3=0.
\end{align}
Next, we argue that $E^4$ is in the null space of $\diamond$. 
To this purpose, we consider the Fourier transform\footnote{note
that its argument $k$ naturally is a cotangent vector} of (\ref{wr01}), that is,
\begin{align}\label{ho20}
\tilde\lambda {\mathcal F}d\phi\otimes k= i{\mathcal F}\psi J
+{\mathcal F}d\sigma\otimes Jk.
\end{align}
Contracting (\ref{ho20}) with $k\otimes k^*$ by the skew symmetry of $J$ yields
$\tilde\lambda |k|^2 k.{\mathcal F}d\phi=0$ and thus\footnote{since $\tr\dot x\otimes\xi=\xi.\dot x$} $\tr {\mathcal F}d\phi\otimes k=0$,
which undoing the Fourier transform implies the first item in (\ref{ho36}).
%
%
By definition (\ref{wr82}) this means
\begin{align}\label{wr31bis}
E^4.\nabla d\phi=0.
\end{align}
In view of definition (\ref{wr51}), this implies
\begin{align}\label{wr31}
E^4 \diamond G = 0\quad\mbox{for all}\;G.
\end{align}

\medskip

Because of stationarity, the covariation of
the pair $(\partial_id\phi^k,\partial_jd\phi^l)$ is invariant under permuting $i$ and $j$.
In particular, the covariation of the pairs
$(\partial_1d\phi^1,\partial_2d\phi^2)$ and $(\partial_2d\phi^1,\partial_1d\phi^2)$
coincide. By definition (\ref{wr51}) this translates into
\begin{align*}
e^1 \otimes e_1 \diamond e^2 \otimes e_2
= e^1 \otimes e_2 \diamond e^2 \otimes e_1.
\end{align*}
Using (\ref{wr82}) to express the $e^i\otimes e_j$'s in terms 
of the $E^n$'s yields
$(E^4+E^1)\diamond(E^4-E^1)$ $=(E^2$ $+E^3)\diamond(E^2-E^3)$,
which by (\ref{wr30}) and (\ref{wr31}) reduces to
\begin{align}\label{wr35}
E^3\diamond E^3=E^1\diamond E^1+E^2\diamond E^2.
\end{align}

\medskip

It remains to identify $E^3\diamond E^3$. To this purpose, we contract (\ref{ho20})
with $J(-ik)\otimes k^*$ to the effect of $\tilde\lambda |k|^2 J(-ik).{\mathcal F}d\phi$
$=i{\mathcal F}d\psi JJ(-ik).k^*$; since $J^2=-{\rm id}$, the latter coincides with  
$-|k|^2{\mathcal F}d\psi$, so that we obtain
$\tilde\lambda \tr {\mathcal F}d\phi\otimes J(-ik)$
$=-{\mathcal F}d\psi$, which in real space amounts to the second item in (\ref{ho36}).
%
%
By definition (\ref{wr82}) we have\footnote{
Since here $J=e_1\otimes e^2-e_2\otimes e^1$, we have $JF=F^te_1\otimes e^2
-F^t e_2 \otimes Fe^1$
and thus $ \tr JF =e_1.Fe^2-e_2.Fe^1$
$= (e^2\otimes e_1-e^1\otimes e_2) . F=- E^3.F$} $\tr JF=-E^3.F$; 
hence the second item in (\ref{ho36}) assumes the form of
\begin{align}\label{ho24}
\tilde\lambda E^3.\nabla d\phi=d\psi.
\end{align}
In view of
definition (\ref{wr51}), this means $E^3\diamond E^3\,d\tilde\lambda^2=[d\psi\,d\psi]$
so that by the normalization of the latter,~cf.~\eqref{in09}, we thus obtain
%
%
%
\begin{align}\label{wr39}
E^3\diamond E^3=1.
\end{align}
Now (\ref{wr52}) and (\ref{wr53}) follow from (\ref{wr48}), (\ref{wr30}),
(\ref{wr31}), (\ref{wr35}), and (\ref{wr39}).


\subsection{Quadratic variation of $\nabla^2d\phi$}\label{ss:var-nabla-nabla-phi}

In this subsection we identify the bilinear form $ \bullet $ 
introduced in definition (\ref{wr80}). The contraction \eqref{ho13}
of $\bullet$ was already identified in \cite[(27)]{SuppCMOW}.

\medskip

Recall that by definition (\ref{wr80}), $\bullet$ is a positive semi-definite
symmetric bilinear form on
the 8-dimensional tensor space
\begin{align}\label{wr84}
\mbox{cotangent space}\otimes\mbox{tangent space}\otimes\mbox{tangent space}.
\end{align}
By isotropy, $ \bullet $ is invariant w.~r.~t.~the natural action of rotations $Q$ 
characterized by
\begin{align}\label{wr78}
Q.(\xi \otimes \dot x\otimes\dot y) :=
(Q^{-t}\xi)\otimes(Q\dot x)\otimes(Q\dot y)
\end{align}
and extended by linearity to (\ref{wr84}). Analogous to Subsection \ref{ss:qv1},
there are natural invariant subspace of (\ref{wr84}) under (\ref{wr78}), namely
\begin{align}
\mbox{trace-free symmetric}&:=\{\textit{\textbf{E}}|\textit{\textbf{E}}^{ij}_k\;
\mbox{invariant under permutations of}\;i,j,k\;\mbox{and}\;\textit{\textbf{E}}^{ij}_i=0\},
\label{ho26}\\
\mbox{isotropic}&:=\{\xi\otimes(e_1\otimes e_1+e_2\otimes e_2)|\xi\;\mbox{cotangent vector}\}
\label{ho27},
\end{align}
which obviously have trivial intersection.
Clearly, (\ref{ho27}) is two-dimensional, and a basis is given by $\{\textit{\textbf{E}}^3,
\textit{\textbf{E}}^4\}$, see (\ref{wr85}) since these tensors can be rewritten as
\begin{align}\label{wr77bis}
\textit{\textbf{E}}^3=e^1 \otimes( e_1\otimes e_1+e_2\otimes e_2)
\quad\mbox{and}\quad
\textit{\textbf{E}}^4=e^2 \otimes( e_2\otimes e_2+e_1\otimes e_1).
\end{align}
As for the dimension of (\ref{ho26}), we note that
the space of (totally) symmetric three-tensors is isomorphic to polynomials (in two variables) 
that are homogeneous of degree 3, and thus is of dimension 4; 
the trace condition means that these polynomials are harmonic,
which reduces the dimension to 2. Hence $\{\textit{\textbf{E}}^1,\textit{\textbf{E}}^2\}$,
see (\ref{wr85}), is a basis of this space\footnote{Note that by the evident
total symmetry of $\textit{\textbf{E}}^1$, $\textit{\textbf{E}}^2$, only one of the trace
conditions has to be checked; it is convenient to contract the first two indices.}.

\medskip

We shall argue that $ \bullet $ is diagonal w.~r.~t.
\begin{align}\label{wr85}
\begin{array}{cl@{\hspace{0.0ex}}l}
\textit{\textbf{E}}^1& \coloneqq \phantom{1}-&e^1\otimes e_1\otimes e_1
+e^2 \otimes e_1\otimes e_2 + e^2\otimes e_2\otimes e_1 + e^1 \otimes e_2\otimes e_2,\\
\textit{\textbf{E}}^2& \coloneqq \phantom{1} -  & e^2 \otimes e_2\otimes e_2  +  e^1 \otimes e_2\otimes e_1  + 
e^1 \otimes e_1\otimes e_2  +  e^2 \otimes e_1\otimes e_1,\\
\textit{\textbf{E}}^3& \coloneqq &e^1 \otimes e_1\otimes e_1+ e^1 \otimes  e_2\otimes e_2,\\
\textit{\textbf{E}}^4& \coloneqq \phantom{1} &e^2 \otimes e_2\otimes e_2 
+ e^2 \otimes  e_1\otimes e_1,\\
\textit{\textbf{E}}^5& \coloneqq \phantom{+} 2&e^1 \otimes e_1\otimes e_1 
+ e^2 \otimes  e_1\otimes e_2 + e^2 \otimes  e_2\otimes e_1,\\
\textit{\textbf{E}}^6& \coloneqq\phantom{+} 2& e^2 \otimes e_2\otimes e_2 +   e^1 \otimes  e_2\otimes e_1 
+ e^1 \otimes  e_1\otimes e_2.
\end{array}
\end{align}
Let us continue to explain why this choice is natural. 
Clearly, (\ref{wr85}) comes in pairs
related by rotation\footnote{
We recall that $J$ is characterized by $Je_1=e_2$ and $Je_2=-e_1$, 
so that its transpose is given by $J^te^2=e^1$ and $J^te^1=-e^2$, 
and therefore the inverse of its transpose behaves like $J$ itself:
$J^{-t}e_1=e_2$ and $^{-t}Je_2=-e_1$.
As a consequence, the action $J.$ on $e^i\otimes e_j\otimes e_k$
replaces an index $1$ by $2$, and an index $2$ by $1$ up to a minus sign.}
by the angle $\frac{\pi}{2}$
\begin{align}\label{wr76}
\textit{\textbf{E}}^2= J.\textit{\textbf{E}}^1,\quad \textit{\textbf{E}}^4
= J.\textit{\textbf{E}}^3,
\quad \textit{\textbf{E}}^6= J.\textit{\textbf{E}}^5.
\end{align}
Also the pair 
$\{\textit{\textbf{E}}^5,\textit{\textbf{E}}^6\}$ is natural: Comparing
(\ref{wr85}) to (\ref{wr82}), we realize\footnote{We recall that the subscript sym
indicates that the identity holds modulo symmetrization of the last two factors
in the triple tensor product.}
\begin{align*}
\textit{\textbf{E}}^5=_{sym} 2 \, E^4\otimes e_1\quad\mbox{and}\quad
\textit{\textbf{E}}^6=_{sym} 2 \,  E^4\otimes e_2,
\end{align*}
which implies
\begin{align*}
\textit{\textbf{E}}^5.\nabla^2d\phi= 2\, \partial_1E^4.\nabla d\phi\quad\mbox{and}\quad
\textit{\textbf{E}}^6.\nabla^2d\phi= 2 \, \partial_2E^4.\nabla d\phi.
\end{align*}
Hence we learn from (\ref{wr31bis}) that
%
\begin{align}\label{wr77}
\textit{\textbf{E}}^5\bullet \textit{\textbf{G}} = \textit{\textbf{E}}^6\bullet \textit{\textbf{G}} = 0\quad\mbox{for all}\; \textit{\textbf{G}}.
\end{align}

\medskip

These 6 elements indeed form a basis of the 6-dimensional 
sub-space relevant for $\bullet$
\begin{align*}
(\mbox{cotangent space})\otimes(\mbox{tangent space})\otimes_{sym}(\mbox{tangent space}),
\end{align*}
which follows from expressing the standard basis as
\begin{equation}\label{wr86}
\begin{aligned}
4 \, e^1 \otimes e_1 \otimes e_1 &= - \textit{\textbf{E}}^1 + \phantom{3\,} \textit{\textbf{E}}^3 + \textit{\textbf{E}}^5, \\
4 \, e^1 \otimes  e_2 \otimes e_2 &= \phantom{-} \textit{\textbf{E}}^1 + 3\, \textit{\textbf{E}}^3 - \textit{\textbf{E}}^5, \\
4 \, e^2 \otimes  e_1 \otimes e_2 
=_{sym} 2 \, ( e^2 \otimes e_1 \otimes e_2 + e^2 \otimes  e_2 \otimes e_1) 
&= \phantom{-} \textit{\textbf{E}}^1 - \phantom{3\,} \textit{\textbf{E}}^3 + \textit{\textbf{E}}^5.
\end{aligned}
\end{equation}
Applying $J.$ to these identities we obtain by definition (\ref{wr78})
and the relation (\ref{wr76}) 
\begin{equation}\label{wr86bis}
\begin{aligned}
4 \, e^2 \otimes e_2 \otimes e_2 &= - \textit{\textbf{E}}^2 + \phantom{3\,} \textit{\textbf{E}}^4 + \textit{\textbf{E}}^6, \\
4 \, e^2 \otimes  e_1 \otimes e_1 &= \phantom{-} \textit{\textbf{E}}^2 + 3\, \textit{\textbf{E}}^4 - \textit{\textbf{E}}^6, \\ 
4 \, e^1 \otimes  e_2 \otimes e_1 
&=_{sym} 
\phantom{-} \textit{\textbf{E}}^2 - \phantom{3\,} \textit{\textbf{E}}^4 + \textit{\textbf{E}}^6.
\end{aligned}
\end{equation}

\medskip

We now give the argument that indeed
\begin{align}\label{wr83}
\textit{\textbf{E}}^m\bullet \textit{\textbf{E}}^n=0\quad\mbox{unless}\;m=n\quad\mbox{and}\quad
\textit{\textbf{E}}^m\bullet \textit{\textbf{E}}^m=\frac{1}{2}\quad\mbox{for}\;m\le 4,
\end{align}
where $\textit{\textbf{E}}^m\bullet \textit{\textbf{E}}^m=0$ for $m\in\{5,6\}$ by (\ref{wr77}). 
We first leverage the invariance of $\bullet$ under the natural action of rotations,
~cf.~(\ref{wr78}). By the argument of Subsection \ref{ss:qv1}
based on the action of rotations on a two-dimensional subspace, which now is (\ref{ho26}), 
and the parameterization of the rotation group via $Q_\theta$, which
now is the rotation by the angle $ -\frac{\theta}{3} $ (as opposed to $\frac{1}{2}\theta$ before), 
so that the first item in (\ref{wr76}) assumes the form of (\ref{ho29}),
and the invariance of a symmetric bilinear form, which now is $\bullet$, we obtain
conformality
\begin{align*}
\textit{\textbf{E}}^1\bullet \textit{\textbf{E}}^1
=\textit{\textbf{E}}^2\bullet \textit{\textbf{E}}^2
\quad\mbox{and}\quad
\textit{\textbf{E}}^1\bullet \textit{\textbf{E}}^2=0,
\end{align*}
and the same statement for $\{1,2\}$ replaced by $\{3,4\}$ and a rotation by 
$ \theta $ instead of $ -\frac{\theta}{3} $.

\medskip

We now give a similar, slightly more involved, argument for
\begin{align}\label{wr79}
\textit{\textbf{E}}^m\bullet \textit{\textbf{E}}^n=0
\quad\mbox{for}\;m\in\{1,2\}\;\mbox{and}\;n\in\{3,4\},
\end{align}
which amounts to showing that the rotation-invariant sub-spaces
(\ref{ho26}) and (\ref{ho27}) are orthogonal w.~r.~t.~$\bullet$.
Using (\ref{wr78}) for $Q=Q_\theta$, applying $\frac{d}{d\theta} \big|_{\theta=0}$,
and appealing to Leibniz' rule and to (\ref{wr48bis}), we obtain\footnote{it
is crucial that the derivative is applied to the action $Q_\theta.$, not just the rotation
$Q_\theta$}
%
\begin{align*}
\frac{d}{d\theta} \Big|_{\theta = 0} Q_\theta.
= -  J^t \otimes  {\rm id} \otimes {\rm id}
+ { \rm id } \otimes J \otimes {\rm id}
+{ \rm id } \otimes  {\rm id} \otimes J
\end{align*}
to which we apply $ J . $, see \eqref{wr78}, which by $ J^2 = - { \rm id } $ and thus by $ J^{-t} = - J^t $ yields
\begin{align*}
J.\frac{d}{d\theta} \Big|_{\theta = 0} Q_\theta.
= - {\rm id} \otimes  J \otimes J
+ J^t \otimes {\rm id}\otimes J
+ J^t \otimes J \otimes {\rm id}.
\end{align*}
We use this to formulate two identities on
${\rm id}+J.\frac{d}{d\theta} \big|_{\theta=0} Q_\theta.$, namely
%
\begin{align*}
({\rm id}+J.\frac{d}{d\theta} \Big|_{\theta = 0} Q_\theta.)
&= {\rm id}\otimes {\rm id} \otimes {\rm id}
-  {\rm id} \otimes J \otimes J
+ J^t \otimes  {\rm id} \otimes J
+ J^t \otimes J \otimes {\rm id},
\end{align*}
which entails
\begin{align*}
({\rm id}+J.\frac{d}{d\theta} \Big|_{\theta = 0} Q_\theta.)
& \big( {\rm id} \otimes {\rm id}\otimes{\rm id}
+{\rm id} \otimes J\otimes J \big) =0.
\end{align*}
Applying these identities to $e^1\otimes e_1\otimes e_1$
we learn by definition (\ref{wr85}) and $ Je_1 = e_2 $, $ J^t e^1 = - e^2 $
\begin{align*}
\textit{\textbf{E}}^1\in {\rm im} \big({\rm id}+J.\frac{d}{d\theta} \Big|_{\theta = 0} Q_\theta.\big)
\quad\mbox{and}\quad
\textit{\textbf{E}}^3\in {\rm ker}\big({\rm id}+J.\frac{d}{d\theta} \Big|_{\theta = 0} Q_\theta.\big).
\end{align*}
Since $J.$ commutes with 
$ {\rm id}+J.\frac{d}{d\theta} \big|_{\theta=0} Q_\theta. $ this entails by (\ref{wr76})
\begin{align*}
\textit{\textbf{E}}^2\in {\rm im} \big({\rm id}+J.\frac{d}{d\theta} \Big|_{\theta = 0} Q_\theta.\big)
\quad\mbox{and}\quad
\textit{\textbf{E}}^4\in {\rm ker}\big({\rm id}+J.\frac{d}{d\theta} \Big|_{\theta = 0} Q_\theta.\big).
\end{align*}
This yields (\ref{wr79}) since 
${\rm id}+J.\frac{d}{d\theta} \big|_{\theta=0} Q_\theta.$ is symmetric w.~r.~t.~$\bullet$,
which follows from the fact that $J.$ is skew 
(since $J.$ preserves $\bullet$ and by $J^2=- {\rm id}$) 
and that $\frac{d}{d\theta} \big|_{\theta=0} Q_\theta.$ is skew (by Leibniz' rule for $\bullet$). 
Hence (\ref{wr79}), and therefore the first item in (\ref{wr83}), are established.

\medskip

Towards the second item in (\ref{wr83}), 
we now claim that by an additional symmetry of $\bullet$ we have
\begin{align}\label{wr81}
\textit{\textbf{E}}^1\bullet \textit{\textbf{E}}^1 - \textit{\textbf{E}}^3 \bullet \textit{\textbf{E}}^3 = \textit{\textbf{E}}^2 \bullet \textit{\textbf{E}}^2 - \textit{\textbf{E}}^4\bullet \textit{\textbf{E}}^4 = 0
\end{align}
so that $ \textit{\textbf{E}}^1\bullet \textit{\textbf{E}}^1 = \textit{\textbf{E}}^2 \bullet \textit{\textbf{E}}^2 = \textit{\textbf{E}}^3 \bullet \textit{\textbf{E}}^3 = \textit{\textbf{E}}^4\bullet \textit{\textbf{E}}^4 $.
Indeed, by stationarity of $d\phi$, which implies that
$[\partial_i\partial_jd\phi\,\partial_k\partial_ld\phi]$
is invariant under permuting $i,j,k,l$,
and by definition (\ref{wr80}) of $\bullet$, the latter
enjoys the additional symmetry 
\begin{equation}\label{wr822}
(\xi \otimes e_i\otimes e_j ) \bullet (\xi' \otimes e_k\otimes e_l )
\quad \mbox{is invariant under permuting}\;i,j,k,l.
\end{equation}
Rewriting the first item in (\ref{wr85}) as
\begin{align*}
\textit{\textbf{E}}^1=-e^1\otimes(e_1\otimes e_1-e_2\otimes e_2)
+e^2\otimes (e_1\otimes e_2+e_2\otimes e_1),
\end{align*}
so that together with the first item in (\ref{wr77bis}) we have
\begin{align*}
\textit{\textbf{E}}^1+\textit{\textbf{E}}^3&=_{sym}\phantom{+}2e^1\otimes e_2\otimes e_2
+2e^2\otimes e_1\otimes e_2,\nonumber\\
\textit{\textbf{E}}^1-\textit{\textbf{E}}^3&=_{sym}-2e^1\otimes e_1\otimes e_1
+2e^2\otimes e_1\otimes e_2,
\end{align*}
we obtain by the symmetry of $\bullet$
\begin{eqnarray}\label{ho22}
\textit{\textbf{E}}^1 \bullet \textit{\textbf{E}}^1 - \textit{\textbf{E}}^3 \bullet \textit{\textbf{E}}^3
&=& - 4 ( e^1 \otimes e_1 \otimes e_1 ) \bullet ( e^1 \otimes  e_2 \otimes e_2 ) 
- 4 ( e^2 \otimes e_1 \otimes e_2 ) \bullet ( e^1 \otimes e_1 \otimes e_1 )\nonumber \\
&& + 4 ( e^1 \otimes e_2 \otimes e_2 ) \bullet ( e^2 \otimes e_1 \otimes e_2 ) 
+ 4 ( e^2 \otimes e_1 \otimes e_2 ) \bullet ( e^2 \otimes e_1 \otimes e_2 ).
\end{eqnarray}
We now appeal to the invariance of $\bullet$ under (\ref{wr78}) with
$Q=J$, in order to rewrite the two last terms in (\ref{ho22}) as
\begin{align*}
( e^1 \otimes e_2 \otimes e_2 ) \bullet ( e^2 \otimes e_1 \otimes e_2 ) 
&= (e^2 \otimes e_1 \otimes e_1 ) \bullet ( e^1 \otimes e_2 \otimes e_1 ),\nonumber\\
( e^2 \otimes e_1 \otimes e_2 ) \bullet ( e^2 \otimes e_1 \otimes e_2 )
&= ( e^1 \otimes e_2 \otimes e_1 ) \bullet ( e^1 \otimes e_2 \otimes e_1 ).
\end{align*}
Hence we learn from
(\ref{wr822}) that the first and last, and the second and third term
in (\ref{ho22}) cancel. This establishes the first equality in (\ref{wr81}). 
In view of (\ref{wr76}), the second equality follows by invariance of $ \bullet $ under the action of $J$.

\medskip

In order to conclude the proof of (\ref{wr83}), we argue
\begin{align}\label{wr88}
\textit{\textbf{E}}^3\bullet \textit{\textbf{E}}^3
+\textit{\textbf{E}}^4\bullet \textit{\textbf{E}}^4=1,
\end{align}
which by (\ref{wr76}) implies $ \textit{\textbf{E}}^3\bullet \textit{\textbf{E}}^3 
= \textit{\textbf{E}}^4\bullet \textit{\textbf{E}}^4 = \frac{1}{2} $. 
To this purpose we introduce
\begin{align}\label{ho25}
\begin{array}{cl@{\hspace{0.0ex}}l}
\textit{\textbf{E}}^7& \coloneqq - & 2 e^1 \otimes  e_2 \otimes e_2 
+ e^2 \otimes e_1\otimes e_2 
+ e^2 \otimes e_2 \otimes e_1, \\
\textit{\textbf{E}}^8& \coloneqq - & 2 e^2 \otimes e_1\otimes e_1  + e^1 \otimes e_2\otimes e_1 
+ e^1 \otimes e_1\otimes e_2,
\end{array}
\end{align}
which are defined such that, see (\ref{wr82}),
\begin{align*}
\textit{\textbf{E}}^7=_{sym}-2 E^3\otimes e_2\quad\mbox{and}\quad
\textit{\textbf{E}}^8=_{sym}-2 E^3\otimes e_1.
\end{align*}
This results in
\begin{align*}
\textit{\textbf{E}}^7.\nabla^2d\phi=2\partial_2 E^3.\nabla d\phi\quad\mbox{and}\quad
\textit{\textbf{E}}^8.\nabla^2d\phi=-2\partial_1 E^3.\nabla d\phi,
\end{align*}
which by (\ref{ho24}) implies
\begin{align*}
\tilde\lambda\textit{\textbf{E}}^7.\nabla^2d\phi=-2\partial_2 d\psi\quad\mbox{and}\quad
\tilde\lambda\textit{\textbf{E}}^8.\nabla^2d\phi=-2\partial_1 d\psi.
\end{align*}
Hence the characterization (\ref{qvarpsi}) of $[\nabla d\psi\cdot\nabla d\psi]$
translates by definition (\ref{wr80}) of $\bullet$ into
\begin{align}\label{wr87}
\textit{\textbf{E}}^7\bullet \textit{\textbf{E}}^7+\textit{\textbf{E}}^8\bullet \textit{\textbf{E}}^8=4.
\end{align}
Because of the representation, cf.~(\ref{wr85}) and (\ref{ho25}),
\begin{align*}
\textit{\textbf{E}}^7=\textit{\textbf{E}}^5-2\textit{\textbf{E}}^3\quad\mbox{and}\quad \textit{\textbf{E}}^8=\textit{\textbf{E}}^6-2\textit{\textbf{E}}^4
\end{align*}
and appealing to (\ref{wr77}), (\ref{wr87}) reduces to (\ref{wr88}).

\section*{Acknowledgements}

The authors thank Princeton University for its warm hospitality in Fall 2023, where this work was started.

\end{document}